\documentclass[]{amsart}
\usepackage{color,latexsym,amssymb,amsmath,amsthm}
\usepackage{comment}
\usepackage{graphicx}
\def\B#1#2{{#1\choose #2}}

\definecolor{orange}{rgb}{1,0.5,0}

\newtheorem{thm}{Theorem}
\newtheorem{coro}[thm]{Corollary}
\newtheorem{lemma}[thm]{Lemma}
\newtheorem{propo}[thm]{Proposition}
\theoremstyle{definition}
\newtheorem*{defn}{Definition}

\newcommand{\E}{{\rm E}}

\title{An index formula for simple graphs}
\author{Oliver Knill}
\date{Apr 30, 2012}
\address{
        Department of Mathematics \\
        Harvard University \\
        Cambridge, MA, 02138
        }


\begin{document}
\maketitle

\begin{abstract}
We prove that any odd dimensional geometric graph $G=(V,E)$ 
has zero curvature everywhere.
To do so, we prove that for every injective function $f$ on the vertex set 
$V$ of a simple graph the index formula 
$\frac{1}{2} [2-\chi(S(x)) - \chi(B_f(x))]=(i_f(x)+i_{-f}(x))/2=j_f(x)$ holds, where 
$i_f(x)$ is a discrete analogue of the index of the gradient vector field $\nabla f$
and where $B_f(x)$ is a geometric graph defined by $G$ and $f$. 
The Poincar\'e-Hopf formula $\sum_x j_f(x)=\chi(G)$ allows so to express the Euler characteristic 
$\chi(G)$ of $G$ in terms of smaller dimensional graphs defined by 
the unit sphere $S(x)$ and the "hypersurface graphs" $B_f(x)$. 
For odd dimensional geometric graphs, $B_f(x)$ is a geometric graph of dimension ${\rm dim}(G)-2$ and 
$j_f(x)=-\chi(B_f(x))/2=0$ implying $\chi(G)=0$ and zero curvature $K(x)=0$ for all $x$. 
For even dimensional geometric graphs, the formula becomes $j_f(x) = 1-\chi(B_f(x))/2$ and allows 
with Poincar\'e-Hopf to write the Euler characteristic of $G$ as a sum  of the Euler 
characteristic of smaller dimensional graphs. 
The same integral geometric index formula also is valid for compact Riemannian manifolds 
$M$ if $f$ is a Morse function, $S(x)$ is a sufficiently small geodesic sphere around 
$x$ and $B_f(x)= S(x) \cap \{ y \; | \; f(y)=f(x) \; \}$. 
\end{abstract}

\section{Introduction}

While for compact Riemannian manifolds $M$, the curvature $K(x)$ satisfying
the Gauss-Bonnet-Chern relation $\int_M K(x) \; d\nu(x)  = \chi(M)$ is defined only in
even dimensions, the discrete curvature 
$$  K(x) = \sum_{k=0}^{\infty} (-1)^k \frac{V_{k-1}(x)}{k+1} \; , $$
satisfying Gauss-Bonnet is defined for arbitrary simple graphs $G=(V,E)$.
In the definition of curvature, $V_k(x)$ is the number of $K_{k+1}$ subgraphs 
in the sphere $S(x)$ at a vertex $x$ and 
$V_{-1}(x)=1$. With this curvature, the Gauss-Bonnet theorem 
$$ \sum_{x \in V} K(x) = \chi(G) \;  $$
holds \cite{cherngaussbonnet}, 
where $\chi(G)=\sum_{k=0}^{\infty} (-1)^k v_k$ is the {\bf Euler characteristic} of the graph, and where
$v_k$ is the number of $K_{k+1}$ subgraphs of $G$. 
Since for all geometric odd dimensional graphs we have seen, the curvature has shown to be 
constant zero, we wondered whether it is constant zero in general. For $5$-dimensional geometric graphs for example,
the curvature reduces to $K(x) = -E(x)/6 + F(x)/4 - C(x)/6$, where $E(x)$ is the number
of edges, $F(x)$ the number of faces and $C(x)$ the number three dimensional chambers in the
unit sphere $S(x)$ of a vertex $x$. We show here that $K(x)$ is constant $0$ for any odd dimensional graph for 
which the Euler characteristic of spheres is $2$ at every point and for which every $S(x)$ inherits this
geometric property. \\

The key tool to analyze the curvature $K(x)$ is to represent it as an expectation of an index 
\cite{indexexpectation}. 
This integral geometric approach shows that the curvature $K(x)$ of a finite simple graph $G=(V,E)$ 
is equal to the expectation of indices $i_f(x)=1-\chi(S_f^-(x))$ when averaging over a probability 
space of injective scalar functions $f$ on $V$ and where 
$S_f^-(x)= \{ y \in V \; | \; f(y)<f(x) \; \}$. 
This links Poincar\'e-Hopf 
$$  \sum_{v \in V} i_f(v) = \chi(G) $$
with Gauss-Bonnet $\sum_{v \in V} K(v) = \chi(G)$ and
gives a probabilistic interpretation of curvature as expectation $K(x) = 1-\E[\chi(S_f^-(x))]$, 
where $\E[\chi(S_f^-(x))]$ is the average Euler characteristic of the subgraph
$S_f^-(x)$ of $S(x)$, 
when integrating over all injective functions. Proof summaries of 
\cite{cherngaussbonnet,poincarehopf,indexexpectation} are included 
in the Appendix A. \\
 
While Gauss-Bonnet, Poincar\'e-Hopf and index expectation hold in full generality for 
any finite simple graph, we can say more if graphs are geometric in the sense
that they share properties from unit spheres in the continuum. 
Similarly as manifolds appear in different contexts like topological or  differentiable 
manifolds, one can distinguish different types of geometric graphs by imposing properties on unit spheres $S(x)$ or
larger spheres. In the present paper we see what happens if $G$ is geometric in the sense that for all vertices $x$, 
the unit spheres $S(x)$ are $(d-1)$-dimensional geometric graphs with Euler characteristic $1+(-1)^d$.  \\

The topological features of unit spheres in a graph define a discretized local
Euclidean structure on a graph. For triangularizations of manifolds, 
the unit spheres are $(d-1)$-dimensional graphs which share properties of the unit spheres 
in the continuum. Without additional assumptions, the unit spheres can be
arbitrarily complicated because pyramid constructions allow
any given graph to appear as a unit sphere of an other graph. A weak geometric restriction is to 
assume that all the unit spheres $S(x)$ are $d-1$ dimensional, have Euler characteristic $1-(-1)^d$, 
and share the property of being geometric. \\

When trying to prove vanishing curvature for odd dimensional graphs, one is lead to an alternating sum 
$W_1(x)-W_2(x)+W_3(x)-\cdots + (-1)^d W_{d-1}(x)$, where $W_k(x)$ is the number of $k$-dimensional 
complete graphs $K_{k+1}$ in $S(x)$ which contain both vertices in $S^-(x)=\{ y \; | \; f(y)<f(x) \; \}$ and 
$S^+(x)=\{ y \; | \; f(y)>f(x) \; \}$. 
We show that $W_k$ can be interpreted as the number $(k-1)$-dimensional faces in a
$d-2$ dimensional polytop $A_f(x)$, a graph which when completed is a $d-2$ dimensional geometric graph
$B_f(x)$ and therefore has Euler characteristic $0$ by induction if $d$ is odd. 
When starting with a three dimensional graph for example, then $B_f(x)$ is a one dimensional geometric 
graph, which is a finite union of cyclic graphs.  \\

This analysis leads to an index formula which holds for a general simple graph $G$:
\begin{equation}
\label{maingeneralformula}
j_f(x) = \frac{1}{2} [2-\chi(S(x)) - \chi(B_f(x))] \; , 
\end{equation}
where $B_f(x)$ is a graph defined by $G,f$ and $x$. For geometric graphs $G$, the graph $B_f(x)$ is 
geometric too of dimension $d-2$ and the formula simplifies to
\begin{equation}
j_f(x) = \frac{1}{2} [1+(-1)^d - \chi(B_f(x))] \; .
\label{mainformula}
\end{equation}
For odd dimensions $d$, it simplifies further to $j_f(x) = -\chi(B_f(x))/2=0$, 
while for even dimensions $d$, it becomes $j_f(x) = 1-\chi(B_f(x))/2$. \\

The formula~(\ref{mainformula}) holds classically for compact Riemannian manifolds $M$ if
$j_f(x)=[i_f(x)+i_{-f}(x)]/2$ for a Morse function $f$ on $M$,
where $i_f(x)$ is the classical index of the gradient vector field $\nabla f$ at $x$ and where
$S(x)=S_r(x)$ is a sufficiently small geodesic sphere of radius $r$. Here $r$ can depend
on $f$. The continuum version is simpler than in the discrete so that it can be given and
proved in this introduction: \\

{\it Let $(M,g)$ be a compact $d$-dimensional Riemannian manifold and let $f \in C^2(M)$ be 
a Morse function. If $x \in M$ is a critical point of $f$ with critical value $c=f(x)$, then 
for small enough $r>0$, the formula~(\ref{mainformula}) holds with $B_f(x) = S_r(x) \cap f^{-1}(c)$,
Also if $x$ is a regular point then the formula holds for small enough positive $r$ 
and gives $j_f(x)=0$.} \\

{\bf Proof:} 
if $x$ is a critical point of $f$, then $i_f(x) = (-1)^{m(x)}$, where $m(x)$ is the Morse index, 
the number of negative eigenvalues of the Hessian of $f$ at $x$. 
If $d$ is odd, then $B_f(x)$ is a $d-2$ dimensional manifold which must
have Euler characteristic $0$ and at a critical point $x$ we have $j_f(x)=[(-1)^{m(x)}+ (-1)^{d-m(x)}]/2 = 0$.  \\
If $d$ is even, then the Morse lemma tells that in suitable coordinates near $x$ that
the function is $f(x) = -x_1^2 - \dots - x_m^2 + x_{m+1}^2 + \cdots + x_d^2$ and
the level surfaces $B_f(x) = \{ f=0 \; \}$ are locally homeomorphic to quadrics intersected with small
spheres $S_r(x)$. The equations 
$2x_1^2 + \dots + 2 x_{m}^2=r^2,  x_{m+1}^2 + \cdots + x_d^2=r^2$ show that this is
$S_{m-1} \times S_{d-1-m}$ topologically. 
Therefore $\chi(B_f(x))=4$ or $\chi(B_f(x))=0$ depending on whether $m$ is 
odd or even.  This implies that $1-\chi(B_f(x))/2=-1$ if $m$ is odd and that it is equal to $1$ if $m$ is even.  
If $x$ is a regular point for $f$, then $B_f(x)$ is a $d-2$ dimensional sphere which has Euler characteristic $2$
and $1-\chi(B_f)/2=0$. An alternative proof uses a property of Euler characteristic 
and the Poincar\'e-Hopf index formula for $i_f(x)$: since 
$$ S(x) = \{ f(y) \leq f(x) \; \} \cup \{ f(y) = f(x) \; \} \cup \{ f(y) \leq f(x) \; \} \; , $$
we have by the inclusion-exclusion principle 
\begin{eqnarray*}
0=1-(-1)^d &=& \chi(S(x))=\chi(S^-_f(x))-\chi(B_f(x))+\chi(S^+_f(x)) \\
         &=& 1-i_f(x)+1-i_{-f(x)}-\chi(B_f(x))=2-2j(f)-\chi(B_f(x))  \; ,
\end{eqnarray*}
where $S^-_f(x) = \{ y \in S(x) \; | \; f(y) \leq f(x) \; \}$. (In the discrete graph case, the $\leq$ in $S^-_f(x)$ is
equivalent to  $<$ because $f$ is assumed to be injective).
Since $r(x)>0$ can be chosen to be continuous in $x$, there exists by compactness of $M$ for every $f$ an $r_f>0$ so 
that the index formula holds for all $x \in M$ if $0<r<r_f$.  \\

{\bf Examples:}  \\
{\bf 1)} For $d=2$ dimensional surfaces $M$, the space $B_f(x))$ is a discrete set on the circle $S_r(x)$, where $f$ is zero.
If we denote the cardinality with $s_f(x)$, the formula tells $j_f(x) = 1-s_f(x)/2$. 
For a Morse function on the  surface $M$, there are three possibilities: either $f$ is a maximum
or minimum, where $s(x)=0$ and $j_f(x)=1$ or $f$ has a saddle point in which case $s(x)=4$ and $j_f(x)=-1$.  \\
{\bf 2)} For a $3$-dimensional manifold, the formula tells $j_f(x) = -\chi(B_f(x))/2$ and $B_f(x)$ is either empty
or consists of two circles, the intersection of a $2$ dimensional cone with a sphere.
In either case, the Euler characteristic is zero. \\
{\bf 3)} Proceeding inductively, since $B_f(x)$ is $(d-2)$-dimensional, we get inductively 
with Poincar\'e-Hopf that all odd dimensional manifolds have zero Euler characteristic. \\
{\bf 4)} For a four dimensional manifold, where
$j_f(x) = 1-\chi(B_f(x))/2$, we either have that $B_f(x)$ is empty which implies $j_f(x)=1$
or that it is the union of a two $2$ dimensional spheres which implies $j_f(x)=-1$.  Or that
it is a single $2$ dimensional sphere so that $j_f(x) = 0$.  \\

In order to prove formula~(\ref{mainformula}) for graphs
it is necessary to define a discrete version of a {\bf hyper surface}
$G_f = \{ f = 0 \; \}$ in a graph. For any simple graph $G=(V,E)$ and any nonzero 
function $f$ on the vertex set there is such a graph $G_f$ or $G_f$ is empty. 
We see examples for some random graphs
in Figure~\ref{rwkgraph}. The idea is motivated but not equivalent to an idea used in statistical mechanics, 
where "contours" are used to separate regions of spins in a ferromagnetic graph. 
We use the edges of the graph which connect vertices for which $f$ takes different sign. 
In graph theory jargon, the graph $G_f$ is a subgraph of the line graph of $G$.
The mixed edges become the vertices of $G_f$ and are connected if they are 
in a common mixed triangle of $G$.  \\

What will be important for us is to see that if $G$ is geometric, then $G_f$ can be completed 
to become a geometric graph $G_f'$. For a $d$-dimensional geometric graph $G$, the new 
graph $G_f'$ is then $d-1$ dimensional. We will apply this construction to 
hyper surfaces $S(x)_f$ in the unit spheres $S(x)$ of a $d$ dimensional graph 
$G$ and call $B_f(x)$ the completion of $S(x)_f= A_f(x)$. Since $S(x)$ is $(d-1)$-dimensional, 
the graphs $B_f(x)$ are
$(d-2)$-dimensional. We see some pictures of $G_f(x)$ when $G$ is three-dimensional in Figures~\ref{onedimensionalcase},
and where $G$ is four dimensional in Figure~\ref{twodimensionalcase}.  \\

To conclude this introduction with the remark that that it is quite remarkable how close simple graphs
and Riemananian manifolds are. While in the previous papers on Gauss-Bonnet and Poincare-Hopf as well as 
in integral geometric representation of curvature, the discrete story is much simpler, we see here that
in the more geometric part, we have to work harder in the discrete case. The continuous analogue has been
proven above. But we should note that in the continuum, the above analysis relies partly on Morse theory. 

\begin{figure}
\scalebox{0.25}{\includegraphics{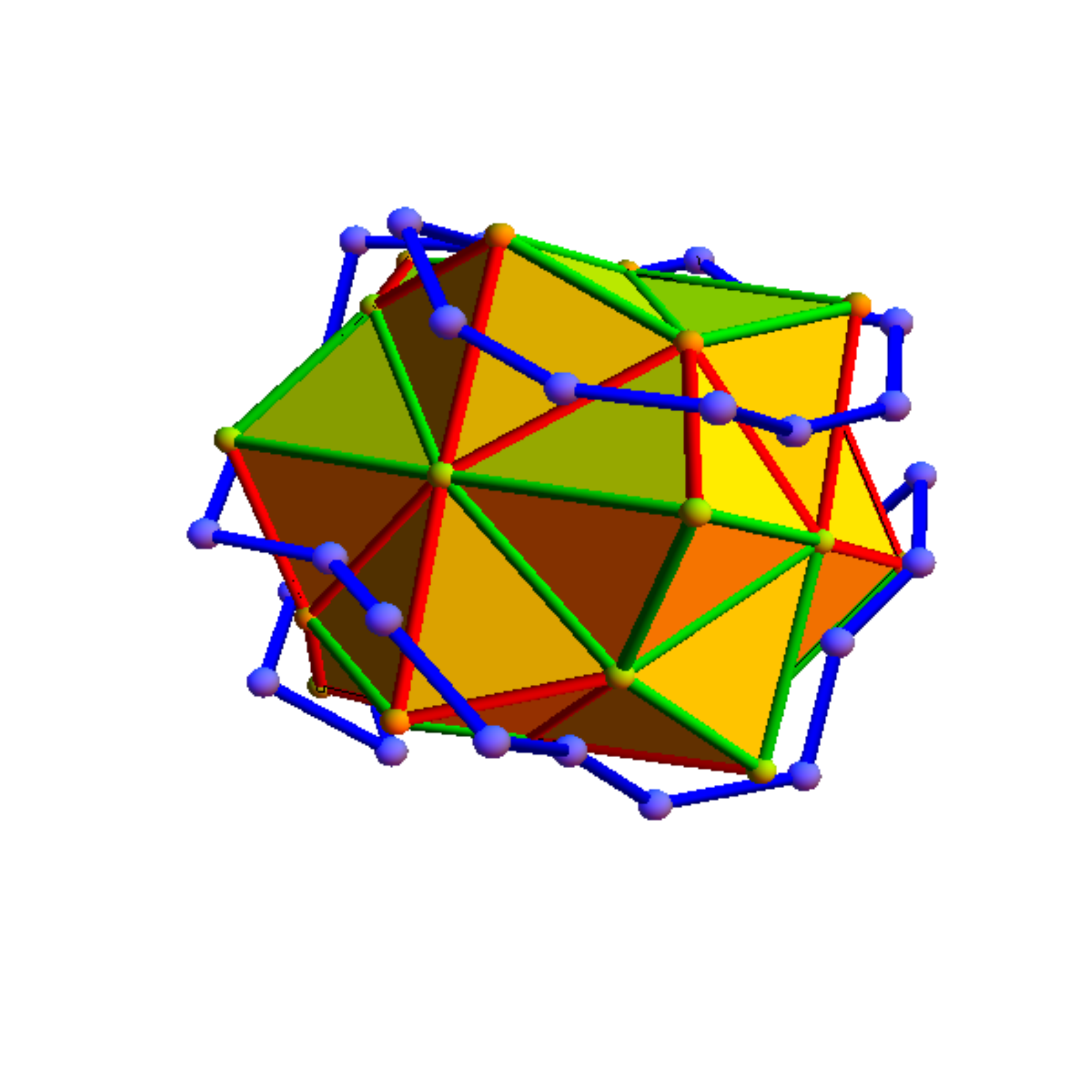}}
\scalebox{0.25}{\includegraphics{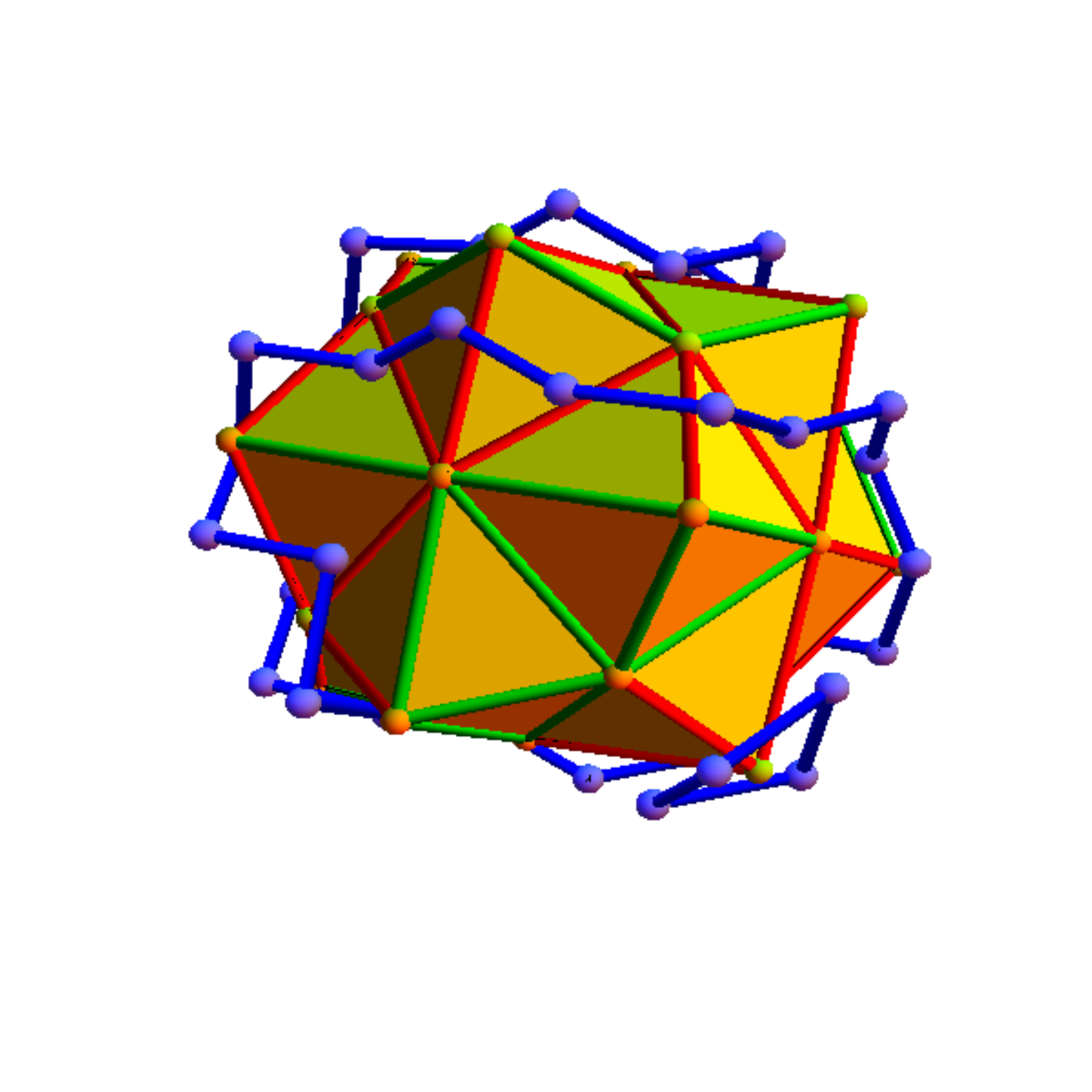}}
\caption{
Two examples of the one-dimensional polytop $B_f(x)$ for two random functions defined by the unit sphere of a 
three dimensional graph $G$. We see the two dimensional unit sphere $S(x)$ of a vertex $x$, 
where triangles have been filled in for clarity and above it the one dimensional graph $B_f(x)$. 
Every vertex of $B_f(x)$ belongs to mixed edges of $S(x)$ and
two vertices in $B_f(x)$ are connected, if they are part of a mixed triangle in $S(x)$. 
For odd dimensional graphs the Euler characteristic of $B_f(x)$ is proportional to the symmetric 
index $j_f(x)$ in the original graph. 
It is zero, implying that the three-dimensional graph $G$ has zero curvature. }
\label{onedimensionalcase}
\end{figure}

\begin{figure}
\scalebox{0.25}{\includegraphics{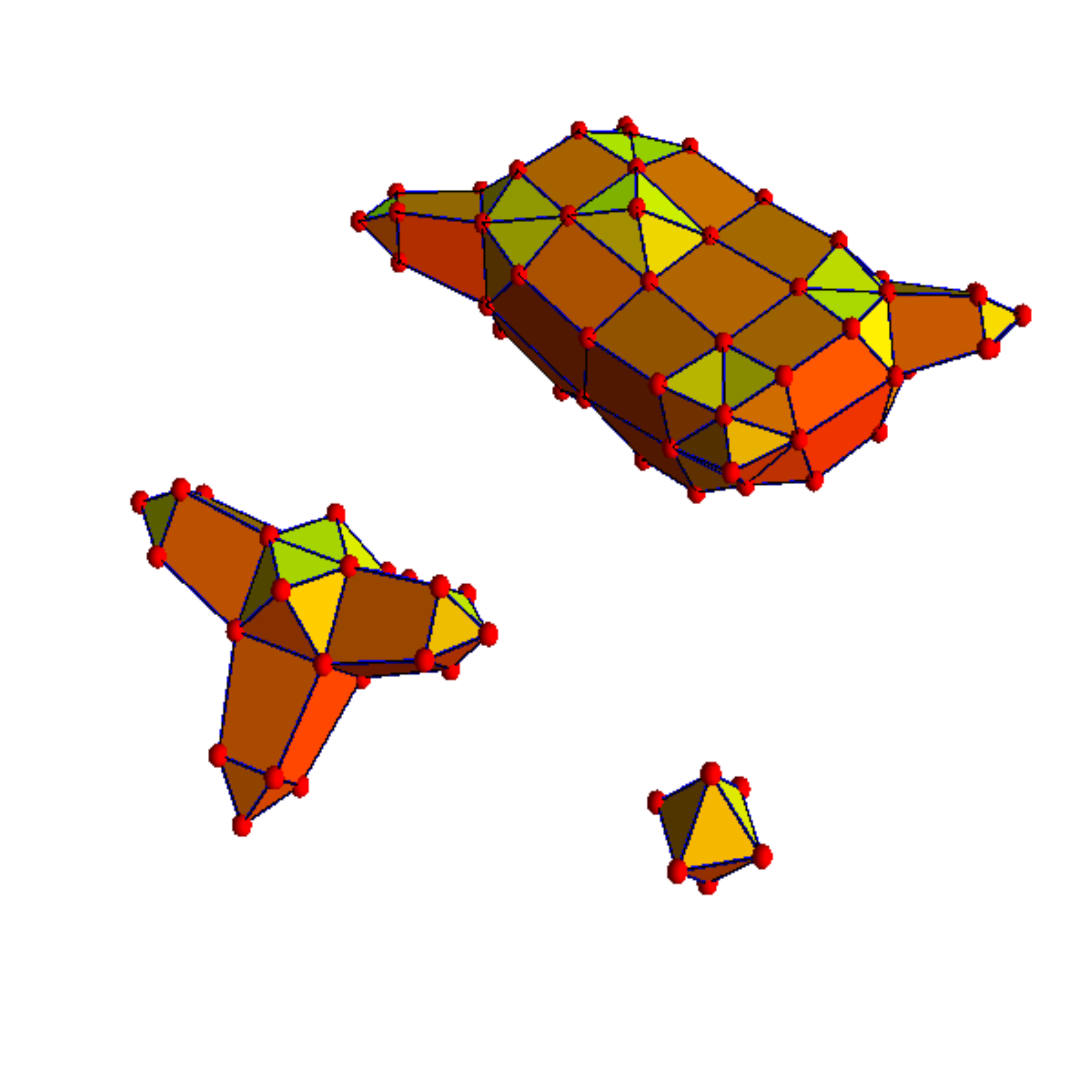}}
\scalebox{0.25}{\includegraphics{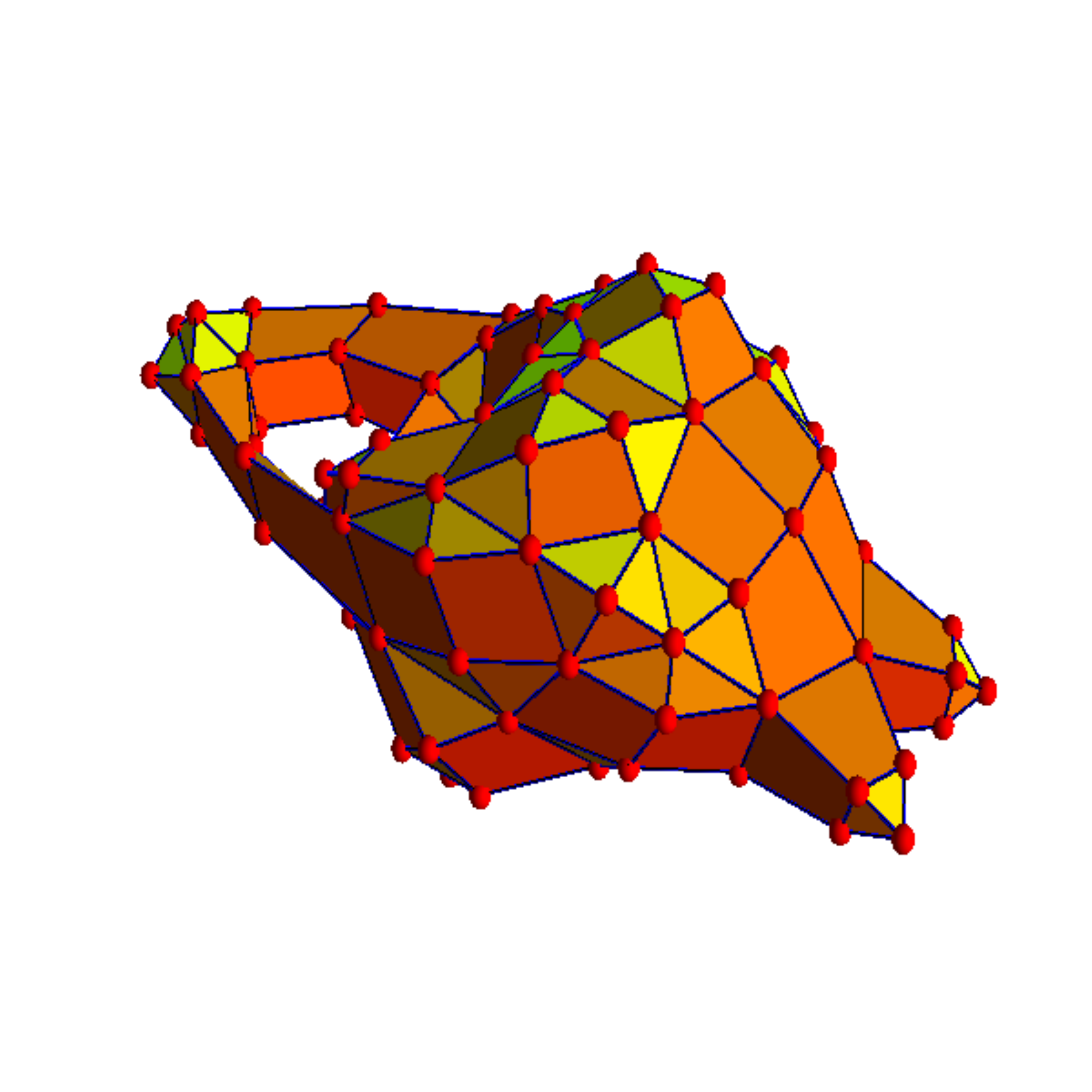}}
\caption{
Two examples of the polytop $B_f(x)$ for random functions defined on a four dimensional graph $G$.
The left example has three components and $\chi(B_f(x))=6$. The right example has one component and a 
hole so that $\chi(B_f(x))=0$. 
The three dimensional unit sphere $S(x)$ of a vertex $x$ is not displayed in this figure.
We only see examples of the "surface" $B_f(x)$. We have $j_f(x) = 1-\chi(B_f(x))/2$ and
the curvature $K(x)$ is the expectation $\E[j_f(x)]$ when integrating over the finite dimensional
parameter space of functions $f$, for which all $f(x_i)$ are independent random variables. }
\label{twodimensionalcase}
\end{figure}

\section{Geometric graphs}

In this section, we define geometric graphs inductively and formulate the main result.
The induction starts with one-dimensional graphs which are called geometric, if every unit 
sphere $S(x)$ in $G$ has 
Euler characteristic $2$. A one-dimensional geometric graph is therefore a finite union of cyclic graphs. 
For a two dimensional geometric graph, each sphere $S(x)$ must be a one dimensional cyclic graph. 
In general: 

\defn{
A graph is called a {\bf geometric graph} of dimension $d$, 
if the dimension ${\rm dim}(x)$ is constant $d$ for every vertex $v$ and
if each unit sphere $S(x)$ is a $(d-1)$-dimensional geometric graph 
satisfying  $\chi(S(x)) = 1-(-1)^d$. A one-dimensional graph is geometric
if every unit sphere has Euler characteristic $2$. \\ }

{\bf Remarks}. \\
{\bf 1)} We could assume connectivity of each sphere $S(x)$ for $d \geq 2$ in order to avoid Hensel type 
situations like gluing
two icosahedra at a single vertex. The unit sphere at this vertex would then consist of two 
disjoint cyclic graphs. But this does no matter for us in this paper, so that we do not require connectivity of $S(x)$.  \\
{\bf 2)} Geometric graphs resemble topological manifolds because their unit spheres have the same Euler characteristic as 
small spheres in a $d$-dimensional manifold. The octahedron and icosahedron are geometric graphs of dimension $d=2$.
While not pursued in this paper, it can be useful to strengthen the assumption a bit and assume that the unit spheres share more
properties from the continuum. Euler characteristic, dimension and connectedness are not sufficient. The following 
definition is motivated by the classical Reeb's theorem: 
a graph $G$ is called {\bf sphere-like} if it is the empty graph or if the minimal number $m(G)$ of critical points among all Morse 
functions is $2$ and inductively every unit sphere $S(x)$ in $G$ is sphere-like. 
Inductively, a geometric graph is called {\bf manifold like} if every unit sphere $S(x)$ is sphere like and manifold like.  
In dimensions $d \geq 4$, there are geometric graphs which are not manifold-like because their unit spheres can have the 
correct Euler characteristic different from a sphere. For us, it is enough to assume that the graph is geometric in the
sense hat we fix the dimension and the Euler characteristic as well as assume that $S(x)$ is 
geometric for every vertex $x \in V$. \\
{\bf 3)} A cyclic graph $C_n$ is geometric for $n>3$. The complete graphs $K_n$ are not geometric because i
the unit spheres have Euler characteristic $1$. From the 5 platonic solids, the octahedron and icosahedron are geometric.
The 600 cell is an example of a  three dimensional sphere type graph and so geometric 
because every unit sphere is an icosahedron which is sphere like. 
We can place $G$ into $R^4$ with an injective function $r: V \to R^4$. For most unit vectors $v$ the function $f(x) = v \cdot r(x)$
is injective with two critical points and $f$ restricted to the unit spheres has the same properties.

\begin{thm}[Zero index for odd dimensional geometric graphs]
For any odd-dimensional geometric graph $G=(V,E)$ and any injective function $f$ on $V$, 
the symmetric index function $j_f(x)$ is zero for all $x \in V$.
\label{zeroindex}
\end{thm}

The proof of Theorem~\ref{zeroindex} is given later after polytopes are introduced. 
The theorem implies the main goal of this article: 

\begin{coro}
Any odd dimensional geometric graph $G$ has constant zero curvature $K(x)=0$ and zero Euler 
characteristic.
\end{coro}

\begin{proof}
From the fact that curvature is equal to the index expectation \cite{indexexpectation} and since 
$j_f(x)$ is zero,  the curvature is zero. 
From Gauss-Bonnet or Poincar\'e-Hopf, we see that the Euler characteristic is zero.
\end{proof}

{\bf Remarks}.  \\
{\bf 1)} The result implies that for odd dimensional geometric graphs with boundary, the curvature is confined
to the boundary. For a one dimensional interval graph $I_n$ for example 
with $v=n+1$ vertices and $e=n$ edges with boundary $v_0,v_n$
the curvature is $1/2$ at the boundary points and adds up to $\chi(I_n)=v-e=1$. \\
{\bf 2)} The previous remark also shows that for odd dimensional orbigraphs, geometric graphs modulo an equivalence relation 
on vertices given by a finite group of graph automorphisms, the curvature is not necessarily zero in odd dimensions. 
Geometric $d$ dimensional graphs with $d-1$
dimensional boundary are special orbigraphs because the interior can be mirrored onto the "other side" leading to an orbigraph
defined by a $Z_2$ graph automorphism action. Because orbigraphs are just special simple graphs, 
curvature and Euler characteristic are defined as before and Gauss-Bonnet of holds as it is true for any simple graph. \\
{\bf 3)} The zero Euler characteristic result is known for triangularizations of odd dimensional 
manifolds by Poincar\'e duality or Morse theory. There are higher dimensional geometric graphs in the sense 
defined here which are not triangularizations of odd dimensional manifolds. 
For such exotic geometric graphs, the unit sphere has the correct
dimension and Euler characteristic but is not a triangularization of a sphere. 
For such geometric graphs, zero Euler characteristic does not directly follow from the continuum. 

\section{Polytopes}

Graphs like the cube or dodecahedron need to be completed to become $2$-dimensional. 
To do so, one can triangulate some polygonal faces.
A triangulation of a face means that a subgraph which is the boundary of a
ball is replaced with the ball. In two dimensions in particular, this means to replace
geometric cyclic subgraphs using pyramid extensions.  This changes the 
Euler characteristic. We call the polytop Euler characteristic 
the Euler characteristic of the as such completed graph.  \\

The $K_n$ graphs have dimension $n-1$ but are not geometric graphs. They can not be
completed in the above way. It can be desirable therefore to extend the operations on polyhedra a bit.
A tetrahedron $G$ for example is a three dimensional graph because
each unit sphere is a triangle which is two dimensional. 
It is not a geometric graph because the Euler characteristic of the unit sphere is 
always $1$.  If we truncate the vertices and replace each vertex with a triangle, 
we obtain a two dimensional graph which we can completed to become a two dimensional 
geometric graph. Such considerations are necessary when giving a comprehensive 
definition of "polytop" which is graph theoretical and agrees with 
established notions of "polytop". In this paper, graphs $K_n$ are not considered polytopes.

\begin{defn}
Given an arbitrary simple graph $G$, we can define a new graph by adding one vertex $z$ and 
adding vertices connecting $z$ with each vertex $v$ to $G$. The new graph is 
called a {\bf pyramid extension} of $G$.  
\end{defn}

{\bf Remarks.} \\
{\bf 1)} The order of the pyramid extension exceeds the order of $G$ by $1$ and the size of the 
extension exceeds the old size by the old order.  \\
{\bf 2)} The pyramid construction shows that any graph can appear as a unit sphere of a larger graph.

\begin{lemma}
The pyramid construction defines from a $d$-dimensional geometric graph $G$ a $d+1$ dimensional 
geometric graph. The extended graph always has Euler characteristic $1$. 
\end{lemma}
\begin{proof}
The construction changes the cardinalities as follows
$v_0$ changes to $v_0+1$ and $v_1$ changes to $v_1+v_0$, $v_2$ becomes $v_2 +v_1$ etc until 
$v_{d}$ becomes $v_d+v_{d-1}$ and $v_{d+1}=0$ becomes $v_d$. \\ 
Therefore,
$$ \chi(G') = \chi(G) + 1-(v_0-v_1 + \dots  + (-1)^d v_d) = 1 \; . $$
\end{proof}

{\bf Examples}. The pyramid construction of a zero dimensional discrete graph $P_n$ is a star 
graph, a tree. The pyramid construction of a cycle graph $C_n$ is a wheel graph $W_n$.
The pyramid construction of the complete graph $K_n$ is the complete graph $K_{n+1}$. 

\begin{defn}
A graph $G$ is called a {\bf $d$-dimensional polytop}, if we can find a finite sequence
of completion steps such that the end product is a $d$-dimensional geometric graph 
$G'$ called the {\bf completion} of $G$. We require that the end product $G'$ does not depend
on the completion process. 
A single {\bf completion step} takes a $0<k<d$ dimensional geometric subgraph of $G$ and
performs a pyramid construction so that it becomes a $k+1$ dimensional geometric graph. 
\end{defn}

{\bf Remarks.}  \\
{\bf 1)} The subgraphs considered in the completion step need to have dimension smaller than $d$ because
otherwise, we could take any already $d$-dimensional geometric graph and add an other completion step
and get a geometric graph of dimension $d+1$.  \\
{\bf 2)} The completion of a geometric graph is the graph itself. This can be seen by induction
on dimension, because if we complete a subgraph, then we complete each unit sphere in that subgraph. 
In other words, a graph in which we can make a completion step is not geometric yet. \\
{\bf 3)} It follows that the completion of a completed graph $G'$ is equal to $G'$. \\

\begin{defn}
The {\bf polytop Euler characteristic} of a polytop $G$ is defined 
as the Euler characteristic of the completed graph $G'$. 
\end{defn}

{\bf Examples.}  \\
{\bf 1)} The polytop Euler characteristic of the cube is $2$ because the completed graph
has $v=14$ vertices, $e=36$ edges and $f=24$ faces and $v_0-v_1+v_2=v-e+f=2$.
The graph Euler characteristic was $8-12=-4$ due to the lack of triangles.
The $6$ squares are $k=1$-dimensional subgraphs.
The completion has added $6$ faces of Euler characteristic $1$ and we can count 
$f_0-f_1+f_2=8-12+6=2$. This is also an example to Proposition~\ref{propo5}. \\
{\bf 2)} A wheel graph $W_k$ with $k \geq 4$ is a two-dimensional polytop because we can make a pyramid construction 
on the boundary which leads to a polyhedron. For $G=W_4$, the completed graph $G'$ is the  octahedron.
Because the subgraph $C_4$ which was completed is odd dimensional, the completion has increased the Euler
characteristic to $2$. \\
{\bf 3)} A two-dimensional graph $G$ is called a geometric graph with boundary, if a vertex is either
an interior point, where the unit sphere $S(x)$ is a cyclic graph or a boundary point, for
which the unit sphere $S(x)$ is an interval graph. The boundary of $G$ is a one dimensional graph, 
a union of finitely many cyclic graphs. We can replace each of these with a corresponding wheel 
graph to get a geometric graph without boundary or what we call a geometric graph for short. \\
{\bf 4)} The complete graph $K_n$ and especially the triangle $K_3$ is not a polytop in the above sense. 
To enlarge the class of polytopes (but not done in this paper) if we would have to allow first the truncation 
of  corners. \\
{\bf 5)} A hypercube $K_2^4$ initially has $v=16$ vertices and $e=32$ edges but no triangles 
nor tetrahedral parts, so that the graph Euler characteristic is $v-e=-16$. To complete it, replace 
each of the $8$ cubes with its completed two dimensional versions, then add $8$ central points each reaching
out to the $14$ vertices of the completed two dimensional cubes,
to get a $3$ dimensional geometric graph. The polytop Euler characteristic is $v-e+f-c=0$ with
$v=48, e=240, f=384, c=192$. We could have computed the polytop Euler characteristic more quickly
by counting faces. There are $v=16$ vertices, $e=32$ edges, and $24$ square faces and $8$ cubic faces.
The polytop Euler characteristic is again $16-32+24-8=0$.  \\

\section{Product graphs}

In this section we look at graph products of complete graphs $K_n$ and show in Proposition~\ref{productlemma}
that if both have positive dimension, then the product graph is a polytop. 

\begin{defn}
The {\bf graph product} $G \times H$ of two graphs $G,H$ is a graph which has as the vertex set 
the Cartesian product of the vertex sets of $G$ and $H$ and where two new vertices 
$(v_1,w_1)$ and $(v_2,w_2)$ are connected if either \\
   {\bf i)} $v_1=v_2$ and $w_1,w_2$ are connected in $H$ or  \\
   {\bf ii)} $w_1=w_2$ and $v_1,v_2$ are connected in $G$. 
\end{defn}

{\bf Examples.} \\
{\bf 1)} Figure \ref{productgraph} shows for example the prism $K_2 \times K_3$. \\
{\bf 2)} The product graph of two cyclic graphs $C_k \times C_l$ with $k,l \geq 4$ 
is called a grid graph. It is a graph $G$ of constant curvature $-1$ with Euler characteristic
$\chi(G) = -kl$. It can be completed by stellating the square faces 
to become a two dimensional graph with Euler characteristic $0$. 

\begin{figure}
\scalebox{0.25}{\includegraphics{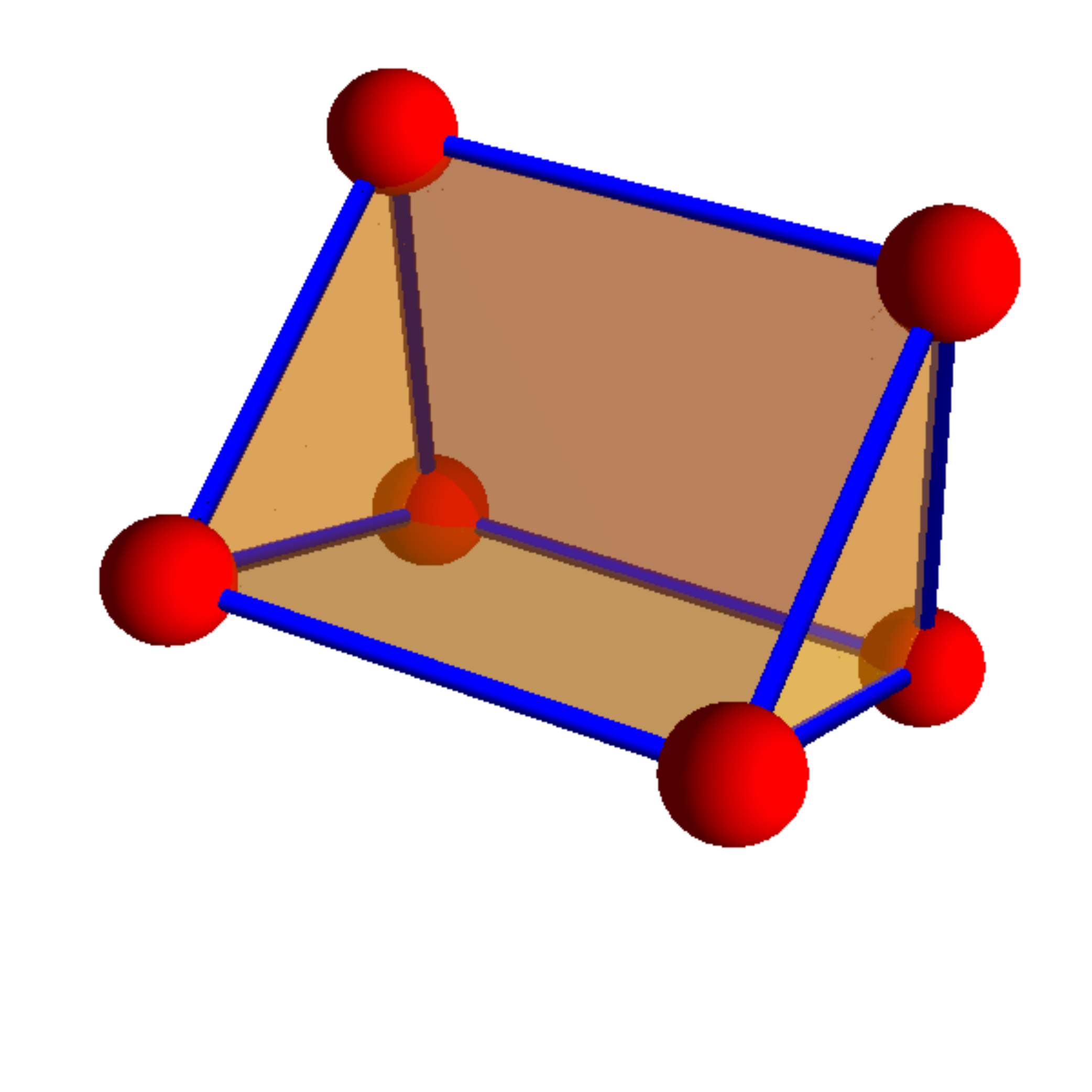}}
\scalebox{0.25}{\includegraphics{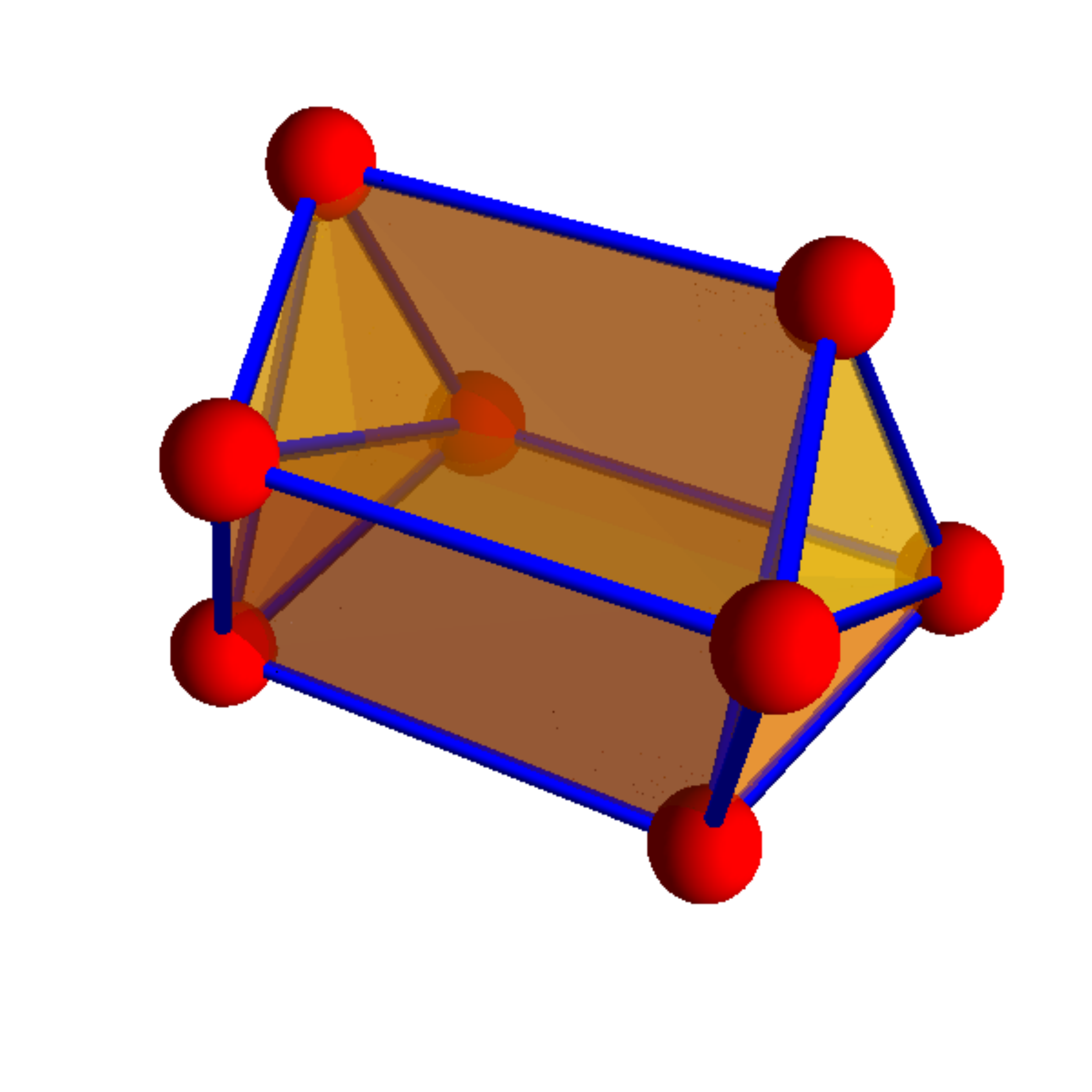}}
\caption{
The left figure shows the graph $K_2 \times K_3$ is a two-dimensional polytop. The 3 faces with 4 vertices can 
be completed. The right figure shows the graph $K_2 \times K_4$ which is a three dimensional polytop. 
To complete it, first complete the quadrilaterals, then complete the three dimensional spaces. 
\label{productgraph}
}
\end{figure}

\defn{
A graph is called a {\bf $d$-dimensional face} if it appears as a unit ball 
$B_1(x)$ of a geometric $d$-dimensional graph $G$. \\} 

{\bf Remark.} \\
Since $B_1(x)$ is the pyramid construction of $S(x)$, a $d$ dimensional face 
always has Euler characteristic $1$.  \\

{\bf Examples}. \\
{\bf 1)} Every complete graph $K_{d+1}$ is a $d$-dimensional face.  \\
{\bf 2)} Given a $d-1$ dimensional geometric graph $G$,
then its pyramid extension $G'$ is a $d$-dimensional face.  \\
{\bf 3)} Every wheel graph $W_n$ is a two dimensional face. Any two dimensional face is either
a triangle $K_3$ or a wheel graph $W_n$ with $n \geq 4$.  \\

The next lemma tells that the product graph of two complete graphs is a polytop and that then an
other pyramid construction of the completion renders it a face. 

\begin{propo}
\label{productlemma}
For $l,m>0$, the completion $K'$ of the graph $K_{l+1} \times K_{m+1}$ is a $d=l+m-1$-dimensional polytop
of graph Euler characteristic $1-(-1)^d$. A further pyramid construction produces a $l+m$-dimensional face, a
graph $\overline{K}'$ of Euler characteristic $1$. If $l=0$ or $m=0$, then $K_{l+1} \times K_{m+1}$ is 
already a $l+m$-dimensional face of Euler characteristic $1$. 
\end{propo}

\begin{proof}
For $l,m=1$, we get $K_2 \times K_2$, where the product is a one dimensional square $C_4$ which is
already geometric. It is a polytop because the completion is unique.  A pyramid construction produces 
the $2$ dimensional wheel graph $W_4$.   \\
To show the claim for general $l,m$ we proceed by induction. Assume it is proven for $(l,m)$ 
it is enough to cover the case $(l+1,m)$ since the other case $(l,m+1)$ is similar.
For any subgraph $H=K_{l+1}$ of $K_{l+2}$, we can by induction uniquely complete $H \times K_{m+1}$ to
become a $l+m$ dimensional face. All these faces are part a $l+m$ dimensional geometric graph. A further
pyramid construction produces a $l+m+1$ dimensional face of Euler characteristic $1$. 
\end{proof}

The examples $K_2 \times K_3$ and $K_2 \times K_4$ are shown in Figure~(\ref{productgraph}).
The graph $K_2 \times K_3$ is a prism, a two dimensional geometric graph. 
A pyramid construction renders it into a three dimensional face. \\

We will now count faces in the $(d-2)$-dimensional geometric graph $B_f(x)$ 
which are obtained from a $d$-dimensional graph $G=(V,E)$ and
a function $f: V \to {\bf R}$. The graph $B_f(x)$ is  made of faces obtained from products
$K_k \times K_l$, where $k+l=d-2$.  For example, if 
$G$ is four dimensional, then each face in $B_f(x)$ consists of triangles $K_3 \times K_1$ and squares 
$K_2 \times K_2$. The dimension of $B_f(x)$ is $2$ and the Euler characteristic of each face is 
$1$. \\

\begin{figure}
\scalebox{0.25}{\includegraphics{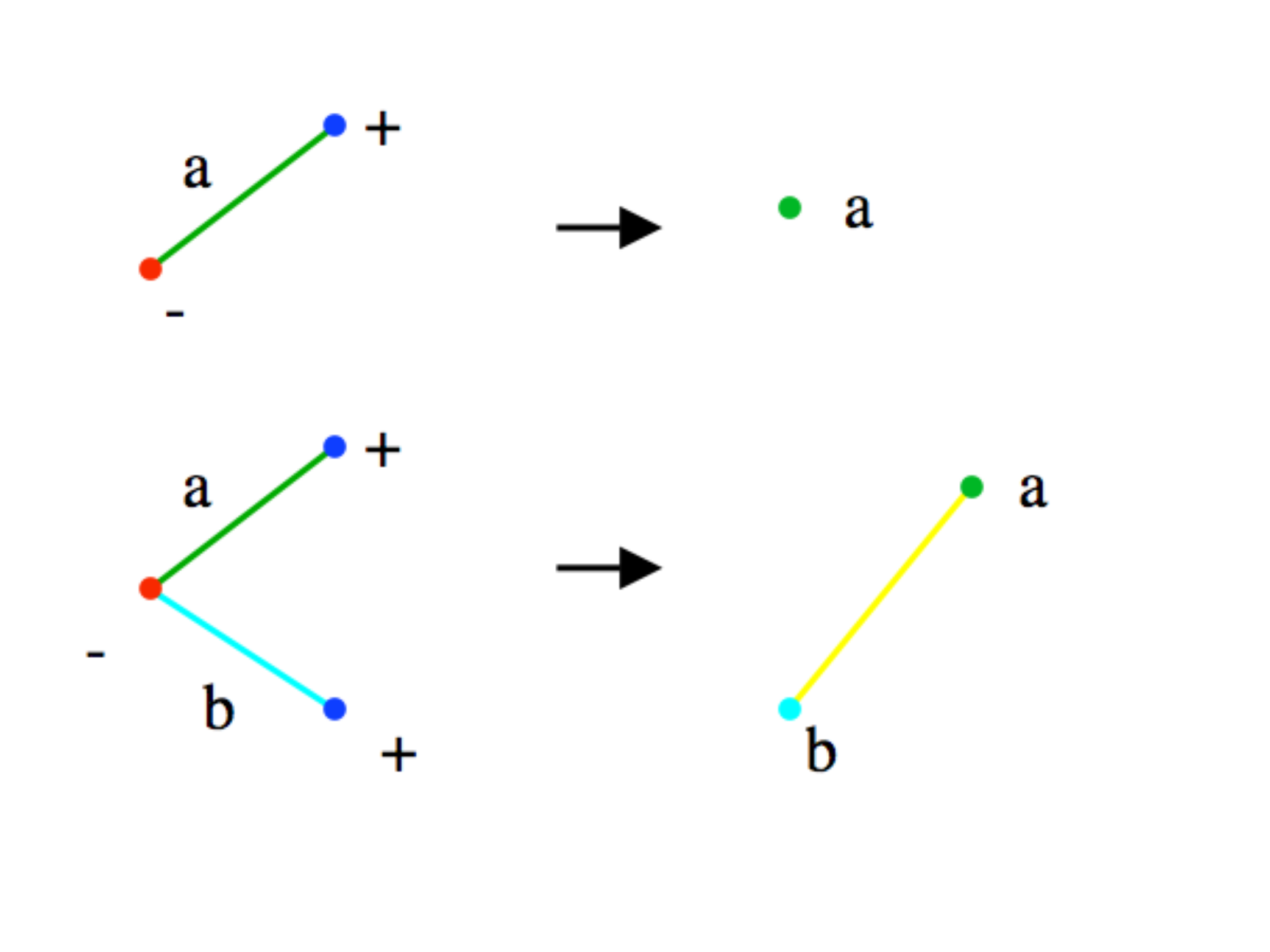}}
\scalebox{0.25}{\includegraphics{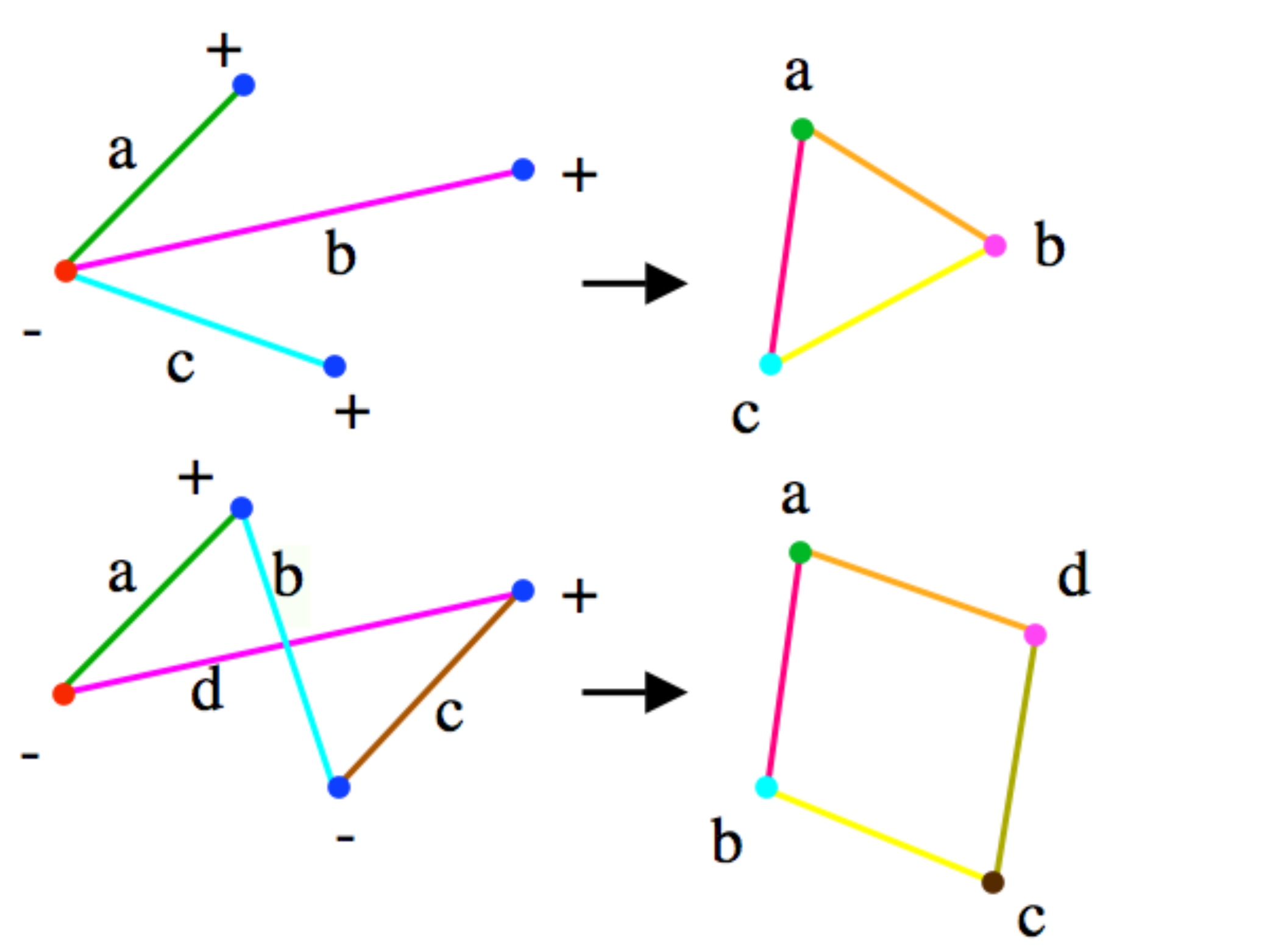}}
\caption{
The left upper figure illustrates that every edge in $G$ defines a point in $G_f$. 
The left lower figure illustrates that a mixed triangle in $G$ defines an edge in $G_f$. 
The upper right figure shows the situation when only one of the vertices in the tetrahedron $K_4$ have a different
sign. This leads to a triangle in $G_f$. The lower right figure shows that if two have positive and 
two have negative sign, then we get a square $K_2 \times K_2$ in $G_f$.
We see in figure~(\ref{twodimensionalcase}) examples of polyhedra which were obtained like that. 
All faces are triangles $K_3 \times K_1=K_3$ or squares $K_2 \times K_2$. }
\label{w2andw3}
\end{figure}

\begin{figure}
\scalebox{0.25}{\includegraphics{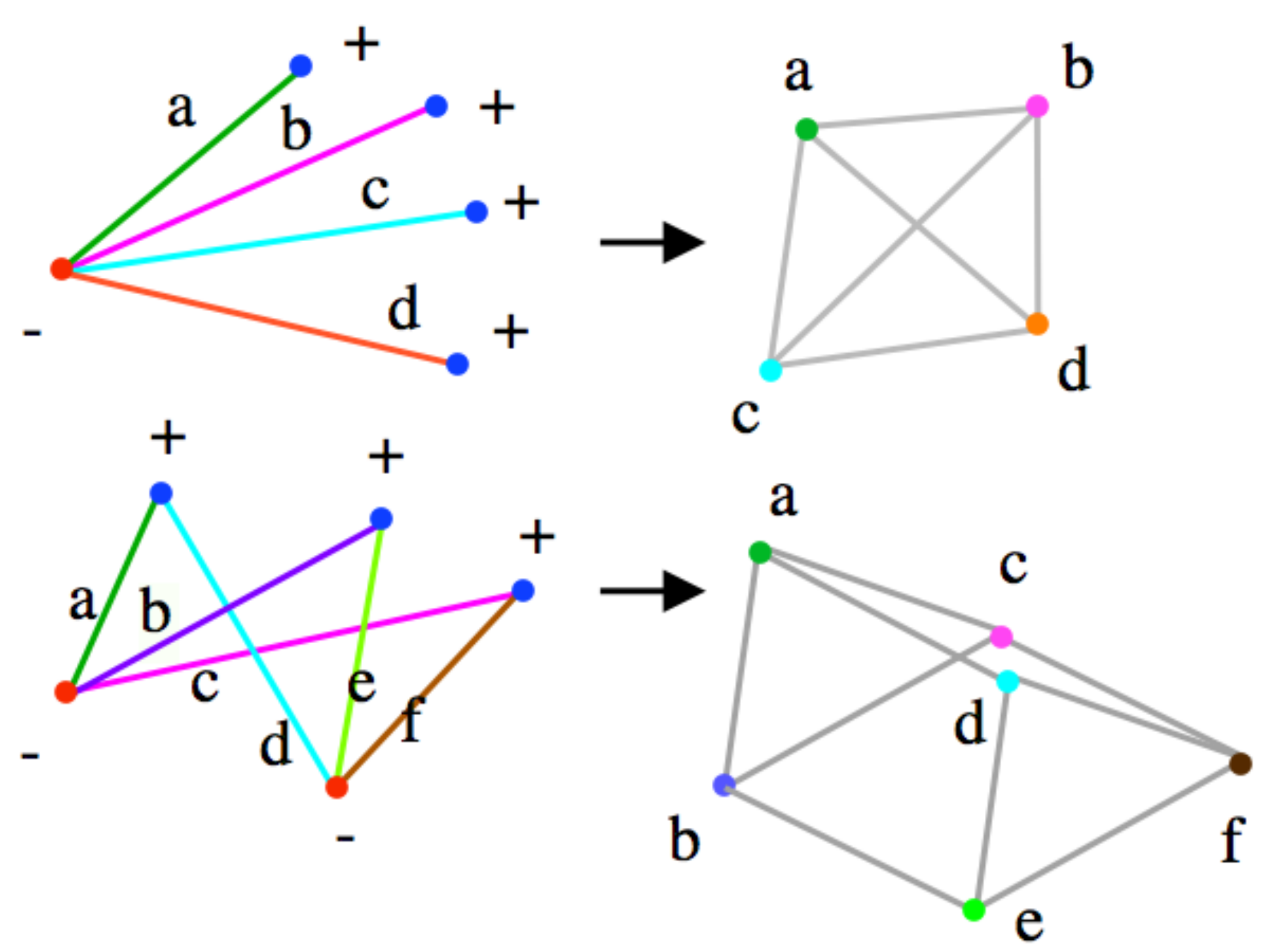}}
\scalebox{0.25}{\includegraphics{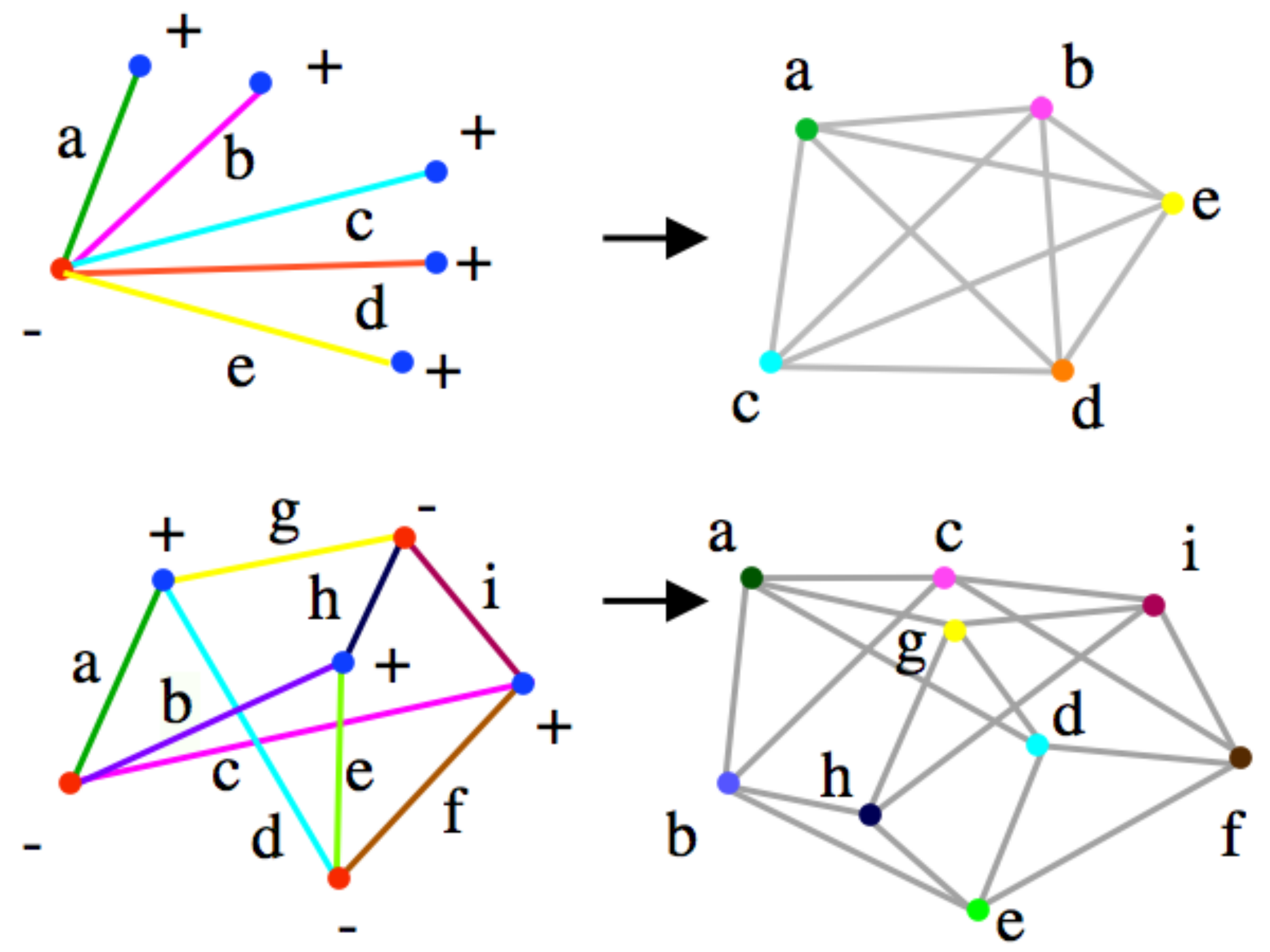}}
\caption{
The $W_4$ subgraphs $K_{5}$ in $S(x)$ define either tetrahedra $K_4 \times K_1$ or
prisms $K_3 \times K_2$.  The $W_5$ subgraphs $K_6$ in $S(x)$ define either 
$K_5 \times K_1$ graphs or prisms $K_4 \times K_2$ (not shown) or a product $K_3 \times K_3$
which is a graph of order $9$. }
\label{w2andw3}
\end{figure}

If $G$ is a polytop with completion $G'$, the polytop Euler characteristic of 
$G$ was defined as $\chi(G')$. The later can also be computed by counting faces: 

\begin{propo}
Assume $G$ is a polytop with completion $G'$ and assume that $G'$ is even
dimensional. If $f_k$ is the number of $k$ dimensional faces in $G'$, then 
$$  \chi(G') = \sum_{i=0}^{\infty} (-1)^k f_k \; . $$
\label{propo5}
\end{propo}
\begin{proof}
The Euler characteristic $\chi(G')$ is a sum $\sum_k (-1)^k v_k$, where $v_k$ counts 
the complete subgraphs $K_{k+1}$ of $G'$. 
Every $k$-dimensional face has Euler characteristic $1$. Counting
the $K_{k+1}$ subgraphs of the completed graph $G'$ gives the same 
result as counting the number $f_k$ of $k$-dimensional faces 
with $(-1)^k f_k$. When writing the Euler characteristic $1$ of
each face as an alternating sum of cardinalities of complete subgraphs,
we double count each boundary of a face. But each of these boundaries are geometric $2d-1$ 
dimensional graphs with Euler characteristic $0$ if $G$ was even
dimensional. 
\end{proof}

{\bf Example:} the dodecahedron $G$ has the graph Euler characteristic $\sum_{k=0}^{\infty} (-1)^k v_k = 20-30=-10$ 
because there are no triangles. The graph $G$ is a "sphere with 12 holes". 
The polytop Euler characteristic is $32-90+60=2$ because the two dimensional completed graph 
$G'$ has $20+12=32$ vertices, $30+60=90$ edges and $5 \cdot 12=60$ triangles so that the 
polytop Euler characteristic is $\chi(G')=2$. This is the Euler characteristic which Descartes has
counted: $f_0=20,f_1=30,f_2=12$ and $\chi(G')=20-30+12=2$.
In this example, each face is a wheel graph $W_5$ with $6$ vertices, $10$ edges and $5$ triangle.
Each double counted boundary is the cyclic graph $C_5$. \\

This concludes the {\bf graph theoretical definitions} of dimension, geometric graph, completion, polytop,
face, Euler characteristic and polytop Euler characteristic.

\section{The graph $G_f$}

In this section we define for a general graph $G=(V,E)$ and any function $f:V \to R$ 
a new graph $G_f$. We will see that if $G$ is geometric of dimension $d$ then $G_f$ has a completion which 
is geometric of dimension $d-1$. The graph $G_f$ is a discrete analogue of 
the "hyper surface" $\{ f=0 \; \} = f^{-1}(0)$ in a Riemannian manifold $M$. The vertices of $G_f$ consist of edges of $G$, 
where $f$ changes sign. The edges of $G_f$ are triangles of $G$ in which $f$ takes different
values. We will see that while the new graph $G_f$ is not geometric, it can be completed to become geometric.  \\

Given a graph $G=(V,E)$ and a function $f:V \to R$ which is nowhere zero, we can
partition the vertex set $V$ into two sets $V^+_f =  \{ x \; | \; f(x)>0 \; \}$ and 
$V^-_f= \{ x \; | \; f(x)<0 \; \}$ of cardinality $s$ and $t$.

\begin{defn}
A subgraph $H$ of $G$ is called {\bf mixed}, if it contains both vertices from $V^+_f$ and $V^-_f$. 
\end{defn}

Two mixed edges in a triangle forces the triangle to be mixed and the third edge to be not mixed.
Every mixed triangle of $G$ therefore has exactly $2$ edges in $V^+_f$ or two edges in $V^-_f$. 

\begin{defn}
Given an simple graph $G=(V,E)$ and a nonzero function $f$ on $V$, a new 
{\bf hypersurface graph} $G_f$ is defined as follows:
\begin{itemize}
   \item The vertices of $G_f$ are the mixed edges of $G$. 
   \item The edges of $G_f$ are the mixed triangles in $G$. 
\end{itemize}
Each mixed triangle connects exactly two mixed edges. 
\end{defn}

{\bf Examples:} \\
{\bf 1)} If $G$ is the octahedron and the function is positive on two antipodal points
and negative everywhere else, then $G_f$ is the union of two cyclic graphs $C_4$. If $f$
is positive everywhere or negative everywhere, then $G_f$ is empty. If $f$ is positive only 
on one vertex, then $G_f$ is a cyclic graph $C_4$. We can also realize $C_6$. 
Because $G$ is a Hamiltonian graph, we can find a coloring so that $G_f = C_8$.  \\
{\bf 2)} If $G=K_3$ and $s=1$ or $s=2$ then $G_f=K_2 \times K_1$  which is $K_2$. \\
{\bf 3)} If $G=K_4$ and $s=1$ or $s=3$ then $G_f=K_3$. If $s=t=2$, then $G_f = K_2 \times K_2$. \\
{\bf 4)} If $G=K_5$ then we have a tetrahedron with $4=\B{4}{1}$ vertices if $s=1$ or $t=1$.
If $s=2$ or $s=3$ we have a prism with $6$ vertices.  \\
{\bf 5)} For $G=K_6$ and $s=2$ or $s=4$ we have $8$ vertices in $B_f$. It is a prismatic graph
connecting two tetrahedra. If $s=3$, we have $9$ vertices and $18$ edges. This is $K_3 \times K_3$. \\

The next lemma tells that any hypersurface graph $G_f$ of a complete graph $K_k$ is a product 
of two complete graphs. It is therefore either a complete graph or a polytop as defined in the last 
section. 

\begin{lemma}
For a complete graph $G=K_k=(V,E)$ and any nonzero function $f:V \to R$,
the graph $G_f$ is isomorphic to the graph product $K_s \times K_t$, 
where $s=|V^+_f|$ and $t=|V^-_f|$. The graph $G_f$ is a polytop in this case. 
\end{lemma}

\begin{proof}
Use induction with respect to in $k$. 
The induction starts at $K_2$, where for $s=t=1$, the graph $G_f$ is $K_1$. 
Assume the claim is settled for $k$ and all $s+t =k$. Take now $G=K_{k+1}$ and $s,$ satisfying $s+t=k+1$.
To use the induction assumption, we assume $v$ is the additional vertex added to $K_k$ to get $K_{k+1}$.
We can assume $f(v)>0$ because the other case is similar. 
Let $v_1,\dots,v_{s+1}$ be the vertices in $V_+$ and let $w_1,\dots, w_t$ be in $V_-$. 
We want to show that the new graph $G_f$ is $K_{s+1} \times K_t$. The graph $G_{f,k}$ defined 
by $K_k$ is a subgraph of $G_{f,k+1}$ the graph defined by $K_{k+1}$. 
We additionally have got $t$ new vertices $(v,w_j)$ in $G_f$. The order of $G_f$ and $K_{s+1} \times K_t$
are the same. The new graph gets new edges $(v,v_i,w_j)$ which connect all the $s \cdot t$ old vertices $(v_i,w_j)$ 
with the new $t$ vertices $(v,w_j)$. There are no connections between the new vertices
in the same was as $K_s \times K_t$ is extended to $K_{s+1} \times K_t$. 
\end{proof}

\begin{figure}
\scalebox{0.25}{\includegraphics{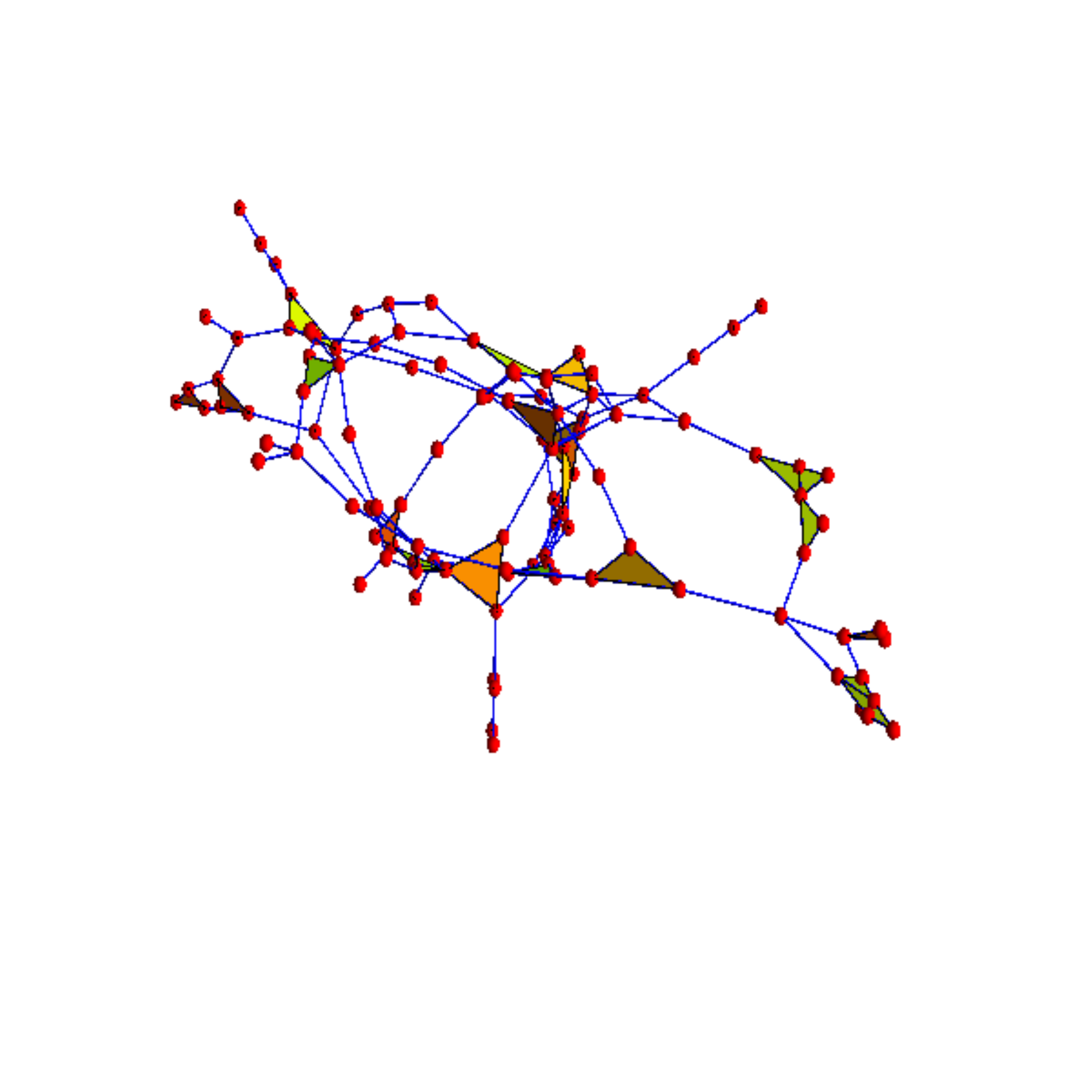}}
\scalebox{0.25}{\includegraphics{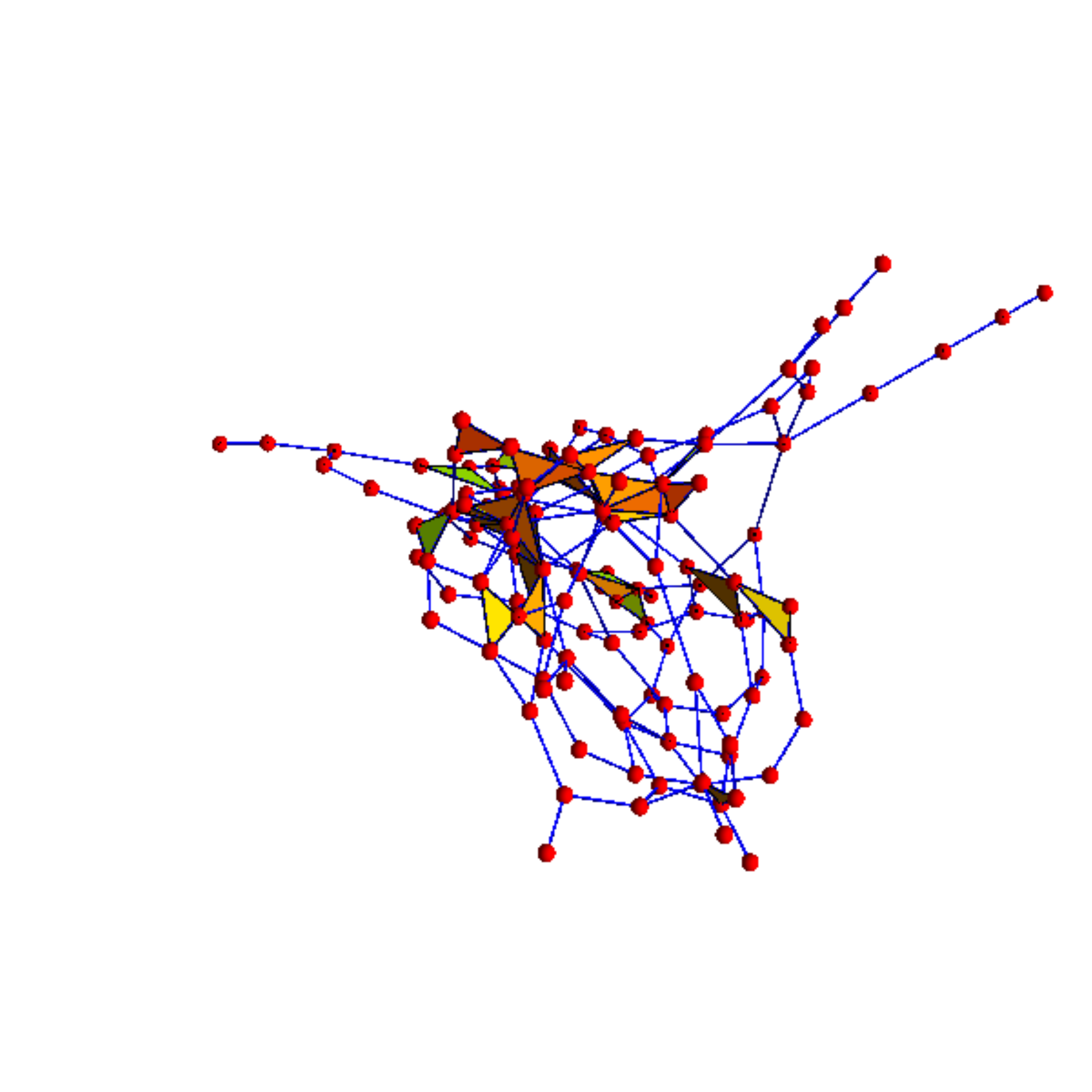}}
\caption{
The figure shows the graph $G_f$ for two random graphs $G$ in the Erd\"os-Renyi 
probability space $G(40,0.3)$. The function $f$ was also chosen randomly. 
The dimension of $G_f$ tends to be smaller than ${\rm dim}(G)-1$. 
For geometric graphs of dimension $d$, the dimension of the completion of $G_f$ is 
exactly $d-1$. 
}
\label{rwkgraph}
\end{figure}

We have now seen that for any $d$-dimensional geometric graph and any nonzero function $f$, 
the graph $G_f$ is a polytop which can be completed to become a $(d-1)$ dimensional geometric graph. 
For $d=1$, this can be seen easily because
a one dimensional geometric graph has no triangles and the graph $G_f$ has no edges and
the dimension of $G_f$ therefore is uniformly $0$. For a two-dimensional geometric graph, 
the graph $G_f$ is a union of closed cycles. Some notation: 

\begin{defn}
For unit spheres $G=S(x)$ and injective $f$ we use the name 
$A_f(x)$ for the graph $S(x)_g$ with $g(y) = f(y)-f(x)$. 
The injectivity of $f$ implies that $g$ is nonzero.
We denote by $B_f(x)$ the completion of $A_f(x)$. 
\end{defn}

\section{The index formula}

The main result in this paper is:

\begin{thm}[Index formula]
If $G=(V,E)$ is a simple graph and $f$ is an injective function on $V$ then
$$  j_f(x)= [2-\chi(S(x))- \chi(B_f(x))]/2 \; . $$
\label{maintheorem}
\end{thm}

The proof of Theorem~\ref{maintheorem} is given below. First a lemma: 

\begin{lemma}[Counting $W_k$]
The number $W_k$ of mixed $k$-dimensional simplices $K_{k+1}$ in $G$ satisfies the formula 
$$ \chi(G_f') = W_1-W_2+W_3-\cdots-W_{2d} \; ,$$
where $G_f'$ is the completion of $G_f$. 
\label{lemmawk}
\end{lemma}

\begin{proof}
We know that each face is of the form $K_s \times K_t$ with $s+t=k$ which can be completed.
The faces of an odd dimensional geometric polytop $G_f$ fit together and form a completed 
polytop $G_f'$. Therefore, the sum $W_1-W_2+W_3-\cdots + (-1)^{2d-1} W_{2d}$ 
the polytop Euler characteristic of $G_f$ which is $\chi(G_f')$. 
See Proposition~\ref{propo5}.
\end{proof}

\begin{coro}
If $G=(V,E)$ is a $d$ dimensional geometric graph, then
$$  \chi(G)=\sum_x (1+(-1)^d)/2 - \chi(B_f(x))/2 \; . $$
Especially, 
$$  \chi(G)= -\sum_{x \in V} \chi(B_f(x))/2 $$ 
for odd dimensional graphs $G$ and
$$  \chi(G)=\sum_{x \in V} 1-\chi(B_f(x))/2 $$ 
for even-dimensional graphs $G$
\label{maincorollary}
\end{coro}

Here is the proof of Theorem~\ref{maintheorem}:

\begin{proof}
Adding
$$  i_f^-(x) = (1-\chi(S^-(x))) \;   $$
and
$$  i_f^+(x) = (1-\chi(S^+(x))) \;   $$
gives
$$ 2j_f(x) = [2 - \chi(S^-(x)) - \chi(S^+(x))]  = 2-\chi(S(x)) - \sum_{k=1}^{\infty} (-1)^k W_k(x)  \; . $$
By Lemma~\ref{lemmawk} applied to the sphere $S(x)$,
we see that $\sum_{k=1}^{\infty} (-1)^k W_k(x)$ is the polytop Euler characteristic of $S(x)_f = A_f(x)$ which
is equal to $\chi(B_f(x))$. 
\end{proof}

The proof of Corollary~\ref{maincorollary} follows immediately:

\begin{proof}
Use Theorem~\ref{maintheorem} and 
the assumption $\chi(S(x)) = 1-(-1)^d$ if $G$ is $d$ dimensional.
\end{proof}

Now the proof of the Theorem~\ref{zeroindex}:

\begin{proof}
We use induction with respect to $d$. 
Assume we know that the symmetric index $j_f(x)$ is zero everywhere for all $(d-2)$-dimensional graphs. 
Then the Euler characteristic of any $d-2$ dimensional geometric graph is zero by Poincar\'e-Hopf
$$ \sum_{x \in V} j_f(x) = \chi(G) \; . $$
Because the dimension $d$ is odd, we have $\chi(S(x))=2$ so that $1-\chi(S(x))/2=0$.
The sum simplifies therefore to
$$ 2j_f(x)= - \sum_{k=1}^{\infty} (-1)^{k} W_k(x)  \;  $$
which is of course a finite sum. 
By Lemma~\ref{lemmawk}, this is the polytop Euler characteristic of a $(d-2)$-dimensional polytop 
and so zero by induction.
\end{proof}

{\bf Examples:}  \\
{\bf 1)} For the 3D cross polytop (the pyramid construction of an octahedron), 
we have $v_0=8,v_1=24,v_2=32,v_3=16$ with $\chi=v_0-v_1+v_2-v_3=0$. 
and $V_0=6, V_1=12, V_2=8$ at every vertex $x$. The curvature by definition is
$$  K(x) = V_{-1}/1 - V_0/2 + V_1/3 - V_2/4 = 1-6/2+12/3-8/4 = 1-3+4-2=0  \; . $$
{\bf 2)} For a 5D cross polytop $G$, with $v_0=12$ vertices, $v_1=60$ edges, 
$v_2=160$ triangles and $v_3=240$ tetrahedra and $v_4=192$ spaces and 
$v_5=64$ halls. The Euler characteristic is $12-60+160-240+192-64=0$. We can compute
$W_k$ data for a typical function $f$
and check that $\sum_{v \in V} \sum_{k=1}^{5} (-1)^{k+1} W_k(v)=0$. 

\section*{Appendix A}

Here are the proofs of the three preprints \cite{cherngaussbonnet,poincarehopf,indexexpectation}.
The index expectation result appears here slightly generalized in that the 
probability measure on functions is allowed to have a general continuous distribution.
Let $v_k$ the number of complete $K_{k+1}$ subgraphs in $G$. Denote by $V_k(x)$ the number of 
$K_{k+1}$ subgraphs in the sphere $S(x)$ with the convention $V_{-1}(x)=1$.

\begin{lemma}[Transfer equations]
\label{transferequations}
$\sum_{x \in V} V_{k-1}(x) = (k+1) v_k$.
\end{lemma}

\begin{proof}
This generalizes Euler's hand shaking lemma $\sum_{x \in V} V_0(x)=2 v_1$:
Draw and count handshakes from every vertex to every center of any $k$-simplex in two 
different ways. A first count sums up all connections leading to
a given vertex, summing then over all vertices leading to $\sum_{x \in V} V_{k-1}(x)$.
A second count is obtained from the fact that every simplex has $k+1$ hands reaching out 
and then sum over the simplices gives $(k+1) v_k$ handshakes.
\end{proof}

Define curvature at a vertex $x$ as $K(x) = \sum_{k=0}^{\infty} (-1)^k \frac{V_{k-1}(x)}{k+1}$.

\begin{thm}[Gauss-Bonnet]
$\sum_{x \in V} K(x) = \chi(G)$.
\label{theorem1}
\end{thm}

\begin{proof}
From the definition $\sum_{x \in V} K(x) = \sum_{x \in V} \sum_{k=0}^{\infty} (-1)^k \frac{V_{k-1}(x)}{k+1}$
we get by (\ref{transferequations})
$$  \sum_{x \in V} K(x) = \sum_{k=0}^{\infty} \sum_{x \in V} (-1)^k \frac{V_{k-1}(x)}{k+1}  
                        = \sum_{k=0}^{\infty} (-1)^k v_k = \chi(G) \;. $$
\end{proof}

Given $f$, let $W_k(x)$ denote the number of all mixed $k$ simplices in the sphere $S(x)$, 
simplices for which $f(y)$ takes both values smaller and larger than $f(x)$. 

\begin{lemma}[Intermediate equations]
$\sum_{x \in V} W_k(x) = k v_{k+1}$
\label{intermediateequations}
\end{lemma}

\begin{proof}
For each of the $v_{k+1}$ simplices $K_{k+2}$ in $G$, 
there are $k$ vertices $x$ which have neighbors in $K_{k+2}$ with both larger and 
smaller values. For each of these $k$ vertices $x$, we can look
at the unit sphere $S(x)$ of $v$. The simplex $K_{k+2}$ defines a 
$k$-dimensional simplex $K_{k+1}$ in that unit sphere. Each of them 
adds to the sum $\sum_{x \in V} W_k(x)$ which consequently is equal to $k v_{k+1}$.
\end{proof}

Define $i_f(x) = 1-\chi(S^-_f(x))$, where $S^-_f(x) = \{ y \in S(x)  \; | \; f(y)<f(x) \; \}$. 

\begin{lemma}[Index stability]
The index sum $\sum_{x \in V} i_f(x)$ is independent of $f$. 
\label{indexstability}
\end{lemma}
\begin{proof}
Deform $f$ at a vertex $x$ so that only for $y \in S(x)$, the value $f(y)-f(x)$ changes from positive to negative.
Now $S^-(x)$ has gained a point $y$ and $S^-(y)$ has lost a point.
To show that $\chi(S^-(x)) + \chi(S^-(y))$ stays constant, show that $V_k^+(x)+ V_k^-(t)$ stays constant,
Since $i(x) = 1- \sum_k (-1)^k V_k^-(x)$, the lemma is proven if $V_k^-(x) + V_k^-(x)$ stays constant. 
Let $U_k(x)$ denote the number of $K_{k+1}$ subgraphs of
$S(x)$ which contain $y$. Similarly, let $U_k(y)$ the number of $K_{k+1}$ subgraphs of
$S(y)$ which do not contain $x$ but are subgraphs of $S^-(y)$ with $x$.
The sum of $K_{k+1}$ graphs of $S^-(x)$ changes by $U_k(y)-U_k(x)$.
Summing this over all vertex pairs $x,y$ gives zero.
\end{proof}

\begin{thm}[Poincar\'e-Hopf]
$\sum_{x \in V} i_f(x) = \chi(G)$.
\end{thm}

\begin{proof}
The number $V_k^-(x)$ of $k$-simplices $K_{k+1}$ in $S^-(x)$
and the number $V_k^+(x)$ of $k$-simplices in $S^+(x)$ are complemented 
within $S(x)$ by the number $W_k(x)$ of mixed $k$-simplices.
By definition, $V_k(x) = W_k(x) + V_k^+(x) + V_k^-(x)$.
By Lemma~\ref{indexstability}, the index $i_f(x)$ is the same for all 
injective functions $f:V \to \mathbf{R}$. Let $\chi'(G) = \sum_{x \in V} i_f(x)$. 
By the symmetry $f \leftrightarrow -f$ switching $S^+ \leftrightarrow S^-$, we 
can prove $2 v_0 - \sum_{x \in V} \chi(S^+(x)) + \chi(S^-(x)) = 2 \chi'(G)$ instead.
Lemma~\ref{transferequations} and Lemma~\ref{intermediateequations} give
\begin{eqnarray*}
 \chi'(G) &=& v_0+\sum_{k=0}^{\infty}(-1)^k \sum_{x \in V} \frac{V_k^-(x) + V_k^+(x)}{2}
           =  v_0+\sum_{k=0}^{\infty}(-1)^k \sum_{x \in V} \frac{V_k(x)   - W_k(x)}{2} \\
          &=& v_0+\sum_{k=0}^{\infty}(-1)^k \frac{(k+2) v_{k+1} - k v_{k+1}}{2}
           =  v_0+\sum_{k=1}^{\infty}(-1)^k v_k = \chi(G) \; .
\end{eqnarray*}
\end{proof}

Define a probability space $\Omega$ of all functions from $V$ to ${\bf R}$, where the 
probability measure is the product measure and where $f(x)$ has a continuous distribution. 
The later assures that injective functions have probability $1$. 
Denote by $\E[X]$ the expectation of a random variable $X$ on $(\Omega,P)$. 

\begin{lemma}[Clique stability]
For all $x \in V$ and all $k \geq 0$,
$$  \E[V_{k-1}^-(x)] = \frac{V_{k-1}(x)}{k+1}  \; . $$
\end{lemma}
\begin{proof}
Since $V_{k-1}(x)$ is the number of $K_{k}$ subgraphs of $S(x)$ and every node is knocked off
with probability $p$, we have $V_{k-1}^-(x) = p^{k} V_{k-1}(x)$. 
The statement in the lemma is proven if we integrate this from $0$ to $1$ 
with respect to $p$. 
\end{proof}

\begin{thm}[Index expectation = curvature]
\label{mainresult}
For every vertex $x$, the expectation of $i_f(x)$ is $K(x)$:
$$  \E[i_f(x)] = K(x) \; . $$
\end{thm}

\begin{proof}
\begin{eqnarray*}
   \E[ 1-\chi(S^-(x)) ] &=& 1-\sum_{k=0}^{\infty}  (-1)^k \E[V_k^-(x)] 
                         = 1+\sum_{k=1}^{\infty} (-1)^k \E[V_{k-1}^-(x)]  \\
                        &=& 1+\sum_{k=1}^{\infty} (-1)^k \frac{V_{k-1}(x)}{(k+1)} 
                         =  \sum_{k=0}^{\infty} (-1)^k \frac{V_{k-1}(x)}{(k+1)} = K(x)   \; .
\end{eqnarray*}
\end{proof}

The curvature function $K(x)$ can depend on the probability measure $P$ of functions but
does not if the function values at different vertices are independent. 

\section*{Appendix B}

Here are some references. \\

Poincar\'e proved the index theorem in the eighth chapter of \cite{poincare85}. 
Hopf extended it to arbitrary dimensions in \cite{hopf26} and mentions that Hadamard had
stated the result without proof in 1910. \cite{Mil65} mentions some contribution of Brouwer.
The result has been generalized by Morse to manifolds with boundary, where the vector field is directed inwards
at the boundary \cite{Morse29} and generalized further
in \cite{Pugh1968,Jubin}. More about the history of Poincar\'e-Hopf is contained in 
\cite{Spivak1999,Mil65,Neill,Gottlieb,Hirsch}.  \\
 
The index of a vector field traces back to a winding number defined by Cauchy in 1831 for 
holomorphic functions. Cauchy in 1937 generalized the notion to differentiable self maps of the plane. 
Kronecker generalized the index to $C^1$ maps of $R^n$ in 1869 cite{DincaMawhin}.
The index of a vector field at an isolated singularity $p$ is traditionally defined as the degree
of the self map $x \to F(p+\epsilon x)/|F(p+\epsilon x)|$ on the sphere.
The degree of a continuous map $f:N \to N$ was first defined topologically in \cite{brouwer11}
and can be traced back to Kronecker \cite{Gottlieb}.
It  can for smooth maps be computed as $\sum_{y \in f^{-1}(x)} {\rm sign}({\rm det} (df(y)))$ at a regular value $x$. \\

The Euler curvature for Riemannian manifolds 
was first defined by Hopf in his 1925 thesis for hypersurfaces \cite{hopf26,Chern1990}
where it is the product of the product n principal curvatures. The general definition is due to 
Allendoerfer \cite{Allendoerfer} and independently by \cite{Fenchel} for Riemannian manifolds embedded in an Euclidean space.
Since we know today that this includes all Riemannian manifolds, the Allendoerfer-Fenchel is the general definition. 
It is also called Lipschitz-Killing curvature \cite{Trofimov} or Euler curvature. It is now part of a larger
frame work, the Euler class being one of the characteristic classes.  \\

The first higher dimensional generalization of Gauss-Bonnet was obtained by Hopf 
(and earlier by Dyck in a special case in 1888 \cite{Petersen,Gottlieb}) for compact hypersurfaces in $R^{2n+1}$
who showed that the degree of the normal map of a hypersurface 
is $\chi(M)/2$ \cite{Spivak1999}, a formula which can be seen topological since it does not involve curvature. 
Gauss-Bonnet in higher dimensions was obtained first independently by Allendoerfer \cite{Allendoerfer}
and Fenchel \cite{Fenchel} for surfaces in Euclidean space using tube methods and extended jointly by
Allendoerfer and Weil \cite{AllendoerferWeil} to closed Riemannian manifolds which from modern perspective
is the full theorem already thanks to the Nash's embedding theorem (i.e. \cite{Gunther}). 
Chern gave the first intrinsic proof in \cite{Chern44}. The proof is given in
\cite{Gray,Spivak1999,Petersen}. See \cite{Cycon,Rosenberg} for a textbook versions of Patodi's proof
using Witten's deformed Laplacian. Gauss-Bonnet has a long history. 
Special cases (see \cite{FuchsTabachnikov}) are the fact known by ancient Greeks that the sum 
of the angles in a triangle are $\pi$, Descartes lost theorem about the sum of the polyhedral 
curvatures in a convex polyhedron or Legendre's theorem stating that the sum of the angles in a spherical triangle in a 
unit sphere is $\pi$ plus the area of the triangle. 
Gauss generalized this Harriot-Girard theorem \cite{Morgan} to geodesic triangles on a surface proving that
the integral of Gauss curvature over the interior is the angle defect.
For expositions of polyhedral Gauss-Bonnet results see \cite{Polya54,Banchoff70}).
By triangularization, this Gauss-Formula leads to Gauss-Bonnet too \cite{Spivak1999}. \\

Discrete differential geometry is useful in many parts of mathematics like
computational geometry \cite{Devadoss,ComputationalGeometry}, the analysis of
polytopes \cite{Schlafli,coxeter,gruenbaum,symmetries},
integrable systems \cite{BobenkoSuris}, networks \cite{CohenHavlin,newman2010,WeinanLuYao}.
computer graphics or computational-numerical methods \cite{Bobenko,CompElectro2002,YJLG}. \\

For Morse theory see \cite{Mil63,Mil65,morse31} or any differential geometry textbook like
\cite{Guillemin,Spivak1999,Hirsch,Henle1994,docarmo94,Warner,Jost,BergerGostiaux,leeriemannian}.
A discrete Morse theory for cell complexes which is closer to classical differential
geometry is developed in \cite{forman95,forman98}. There, Morse theory is built for 
simplicial complexes. It is close to classical Morse theory and many classical results like
the Reeb's sphere theorem is proven in this framework. 
The theory has concrete applications in graph theory. \\

Discrete differential geometry is used also in 
gravitational physics \cite{GSD,Regge,Misner,Fro81,MarsdenDesbrun} and 
statistical mechanics \cite{Itzykson} and other parts \cite{GSD}. 
The elementary approach to discrete differential geometry followed here is graph 
theoretical \cite{BM,bollobas1,Chung97} and therefore different from the
just mentioned approaches which include more structure, usually structure from Euclidean
embeddings like angles or lengths. These concepts are always related because 
any simplicial complex defines a graph and conversely, graph are embedded 
in Euclidean space to model geometric objects there. Setting things up purely graph theoretically
can reveal how much is combinatorial and independent of Euclidean embeddings.  \\

Discrete curvature is by Higushi traced back to a combinatorial curvature defined in \cite{Gromov87}.
It is there up to normalization defined by $K(p) = 1-\sum_{j \in S(p)} (1/2-1/d_j)$, where 
$d_j$ are the cardinalities of the neighboring face degrees
for vertices $j$ in the sphere $S(p)$. 
For two dimensional graphs, where all faces are triangles,
this simplifies to $d_j=3$ so that $K=1-|S|/6$, where $|S|$ is the cardinality of the 
sphere $S(p)$ of radius 1 appears in \cite{elemente11}. Discrete curvature was used
in \cite{Higuchi} who uses a definition of an unpublished talk by Ishida from 1990. 
Higushi's curvature is used in large-scale data networks \cite{NarayanSaniee}
Gauss-Bonnet for planar graphs appears in the form $\sum_n (6-n) f_n = 12$
in \cite{princetonguide} and in \cite{BM} in the form $\sum_g C(g)=-12$ 
where $C(g) = d(g)-6$ is called the {\bf charge} at $x$, where $d(g)$ is the degree.
Estimates for the size of graphs with positive combinatorial curvature \cite{RetiBitayKosztolanyi}.
The general Euler curvature was first defined in \cite{cherngaussbonnet}. It appears to be the first
curvature defined for general simple graphs which satisfies Gauss-Bonnet.
Other curvatures which refer to Euclidean embeddings are sometimes called discrete curvature 
and used in computer graphics both for surfaces as well as in image analysis 
for numerical general relativity like Regge calculus. \\

Other unrelated curvatures for graphs have been defined: in \cite{beresina}, 
the curvatures of the surface $2z=\langle x,Ax \rangle$ in $R^{{\rm ord}(G)+1}$ 
are considered, where $A$ is the adjacency matrix of $G$. 
The local clustering coefficient defined by Watts and Strogatz \cite{WattsStrogatz} 
is $c(x) = V_1(x)2/(V_0(x)(V_0(x)-1))$. 
Ricci curvatures for Markov chains on metric spaces has been defined by 
\cite{Ollivier} and studied for graphs in \cite{LinLuYau,JostLiu} for pairs of vertices $x,y$.
A distance between probability distributions $m_1,m_2$ on the vertex set $V$ can be defined as
$W(m_1,m_2) = \sup_f \sum_x f(x)[m_1(x)-m_2(x)]$, where the supremum is taken over all $1$-Lipshitz functions $f$.
Now define $k_{\alpha}(x,y)=1-W(m_x,m_y)/d(x,y)$, where $m_x(v)=\alpha$ for $v=x$ and 
$m_x(v)=(1-\alpha)/V_0(x)$ for $v \in S(x)$. The Ollivier Ricci curvature is the limit 
$k(x,y) = \lim_{\alpha \to 1} k_{\alpha}(x,y)$.  \\

The inductive dimension for graphs \cite{elemente11}
is formally related to the inductive Brouwer-Menger-Urysohn dimension for topological spaces
which goes back to the later work of Poincar\'e \cite{HurewiczWallman,Engelking,Pears,munkres} and was given by 
Menger in 1923 and Urysohn in 1925. It also is related to Brouwer's Dimensionsgrad \cite{brouwer1913} given in 1913.
Any natural topology on a graph renders a graph zero-dimensional with the Menger-Urysohn as well as the Brouwer definition.
The Menger-Uryson dimension for example is zero if there is a basis for open sets for which every basis element is 
both closed and open. 
The inductive dimension for graphs was already defined in \cite{elemente11} is a rational number. It is natural also in 
probabilistic contexts as its expectation can be computed easily in random graphs \cite{randomgraph}.
There are other unrelated notions of dimension in graph theory like the EHT embedding dimension
\cite{EHT} which is based on the smallest dimensional Euclidean space in which the graph can be embedded, with variants
like Euclidean dimension \cite{Soifer} or faithful dimension \cite{ErdoesSimonovits}.
The metric dimension of a graph \cite{HararyMelter} 
is the minimal number of points $x_i \in V$ allowing to characterize
every vertex by its distance coordinates $a_i = d(x,x_i)$. The metric dimension of $K_n$ is $n-1$ of a path it is $1$,
for a cyclic graph or tree which is not a path it is larger than $1$. \\

The history of the notion of polyhedron and polytop is complicated and "a surprisingly delicate task"
\cite{Devadoss}. "Agreeing on a suitable definition is surprisingly difficult" \cite{Richeson}.
Coxeter \cite{coxeter} defines it as a convex
body with polygonal faces. Gruenbaum \cite{gruenbaum} also works with convex polytopes, the
convex hull of finitely many points in $R^n$. The dimension is the dimension of its affine span.
The perils of a general definition have been pointed out since Poincar\'e (see \cite{Richeson,cromwell,lakatos}).
Topologists started with new definitions \cite{Alexandroff,Fomenko,ConwayCapstone,Spanier}, 
and define first a simplicial complex and then polyhedra as topological spaces which admit a triangularization by 
a simplicial complex. In this paper, we use a graph theoretical definition slightly modified from 
\cite{elemente11}: a polytop is a finite simple graph which can be completed to become a $d$-dimensional geometric graph. 
Since we would like to include $K_n$ as polytopes, one can add a voluntary truncation process
of vertices first. The graph theoretical definition it reaches all 
uniform polytopes or duals as well as graphs defined by convex hulls of finitely many points in $R^n$
or graphs defined from simplicial complexes. The assumption on the unit sphere of the geometric graph will decide
what polyhedra are included. If we insist for a two dimensional polyedron every unit sphere to be connected
then the Skilling's figure $G$ is not included. If we insist the unit spheres just to be one dimensional graphs with 
zero Euler characteristic, then $G$ is a polytop. \\

Euler characteristic for polytopes was first used by Descartes in a letter of 1630 to Leibnitz \cite{Stillwell2010},
but he did not take the extra step to prove the polyheon formula $v-e+f=2$. 
Euler sketched the first proof of the polyhedra formula in a November 1750 letter to Goldbach \cite{EulerGoldbach} where
he gives a triangularization argument. The limitations of Euler's proof are now well understood 
\cite{lakatos,FranceseRicheson}. In modern language, Euler's formula tells that the combinatorial Euler characteristic $v-e+f$ 
coincides with the cohomological Euler characteristic $b_0-b_1+b_2=2$ for connected, simply connected geometric graphs, where
"geometric" means that  $S(x)$ is a connected cyclic graph at every point and that $G$ can be oriented. 
The connectedness assures $b_0=1$, the orientability implies $b_0=b_2$ and the simply connectedness forces $b_1=0$. 
While we now know that all this can be defined purely graph theoretically without assuming any Euclidean embeddings,
Euler - evenso himself a pioneer in graph theory - did not consider polyhedra as graphs. This was a step only taken
by Cauchy \cite{Malkevich}. The story about the Euler characteristic are told in \cite{Aczel,Richeson,Eckmann1,lakatos}. 
The average Euler characteristic in random graphs was computed in \cite{randomgraph}.  \\

Random methods in geometry is part of integral geometry as pioneered by Crofton and Blaschke 
\cite{Blaschke,Nicoalescu}. 
Integral geometry  has been used in differential geometry extensively by Chern in the form of kinematic 
formulae \cite{Nicoalescu}. This is not surprising since Chern is a student of Blaschke.
It was also used by Milnor in proving total curvature estimates for knots made widely known by 
Spivak's textbook \cite{Spivak1999}. Banchoff used integral geometric methods in \cite{Banchoff67} 
and got analogue results for Polyhedra and surfaces similar to what was obtained in \cite{indexexpectation}
for graphs. Whether integral geometric methods have been used in graph theory before \cite{indexexpectation} 
is unknown to this author. \\

The history of topology \cite{Dieudonne1989} and graph theory \cite{BiggsLloydWilson,HistoryTopology}.
For an introduction to Gauss-Bonnet with historical pointers to early discrete approaches, 
see \cite{Wilson}. More historical remarks about Gauss-Bonnet are in \cite{Chern1990}. 


\vspace{12pt}
\bibliographystyle{plain}

\end{document}